\documentclass[11pt]{article}
\usepackage[utf8]{inputenc}
\usepackage{amsfonts,latexsym,amsthm,amssymb,amsmath,amscd,euscript,mathrsfs,graphicx}
\usepackage{framed}
\usepackage{fullpage}
\usepackage{color}
\usepackage{enumitem}
\usepackage[obeyFinal,textsize=scriptsize,shadow]{todonotes}
\usepackage{tikz}
\usetikzlibrary{matrix}
\usepackage{hyperref, cleveref}
\usepackage{mathtools}

\DeclarePairedDelimiterXPP\EV[1]{E}[]{}{

#1}

\newtheorem{theorem}{Theorem}[section]
\newtheorem{proposition}[theorem]{Proposition}
\newtheorem{lemma}[theorem]{Lemma}
\newtheorem{corollary}[theorem]{Corollary}

\theoremstyle{definition}
\newtheorem{defn}[theorem]{Definition}

\theoremstyle{remark}

\newcommand{\BE}{\mathbb E}

\newcommand{\BI}{\mathbb I}

\newcommand{\BP}{\mathbb P}

\newcommand{\BR}{\mathbb R}

\newcommand{\cA}{\mathcal{A}}
\newcommand{\cD}{\mathcal{D}}
\newcommand{\cN}{\mathcal{N}}
\newcommand{\cO}{\mathcal{O}}
\newcommand{\cX}{\mathcal{X}}

\newcommand{\eps}{\varepsilon}
\newcommand{\indep}{\perp\!\!\!\perp} 
\DeclareMathOperator{\Var}{Var}
\DeclareMathOperator{\Tr}{Tr}
\DeclareMathOperator{\poly}{poly}

\title{Better and Simpler Lower Bounds for Differentially Private Statistical Estimation}
\author{Shyam Narayanan\thanks{Massachusetts Institute of Technology. Email: \texttt{shyamsn@mit.edu}. Research supported by an NSF Graduate Fellowship and a Google Fellowship.}}
\date{\today}

\begin{document}

\maketitle

\vspace{-0.5cm}
\begin{abstract}
    We provide optimal lower bounds for two well-known parameter estimation (also known as statistical estimation) tasks in high dimensions with approximate differential privacy. First, we prove that for any $\alpha \le O(1)$, estimating the covariance of a Gaussian up to spectral error $\alpha$ requires $\tilde{\Omega}\left(\frac{d^{3/2}}{\alpha \eps} + \frac{d}{\alpha^2}\right)$ samples, which is tight up to logarithmic factors. This result improves over previous work which established this for $\alpha \le O\left(\frac{1}{\sqrt{d}}\right)$, and is also simpler than previous work. Next, we prove that estimating the mean of a heavy-tailed distribution with bounded $k$th moments requires $\tilde{\Omega}\left(\frac{d}{\alpha^{k/(k-1)} \eps} + \frac{d}{\alpha^2}\right)$ samples. Previous work for this problem was only able to establish this lower bound against pure differential privacy, or in the special case of $k = 2$.
    
    Our techniques follow the method of fingerprinting and are generally quite simple. Our lower bound for heavy-tailed estimation is based on a black-box reduction from privately estimating identity-covariance Gaussians. Our lower bound for covariance estimation utilizes a Bayesian approach to show that, under an Inverse Wishart prior distribution for the covariance matrix, no private estimator can be accurate even in expectation, without sufficiently many samples.
\end{abstract}

\section{Introduction} \label{sec:intro}

Mean and covariance estimation are two of the most fundamental tasks in algorithmic statistics. Simply put, the goals of these tasks, respectively, are: given i.i.d. samples $X_1, \dots, X_n$ from an unknown distribution $\mathcal{D}$, can we estimate the mean (resp., covariance) of the distribution? This question is especially worthy of investigation for data in high-dimensional Euclidean space, as this setting not only captures many real-world data problems but also has led to numerous theoretically and practically interesting algorithms. 

In many practical use cases, the data samples come from humans, and unfortunately, naive empirical mean and covariance estimates of the data may reveal highly sensitive information about an individual. Hence, one wishes to approximately compute the mean and covariance while protecting the privacy of the individuals that contribute the data. The goal of \emph{provably} protecting privacy in algorithm design led to the notion of differential privacy~\cite{DP}, which has become the gold standard of ensuring privacy both in theory and in practice. Formally, differential privacy is defined as follows.
\begin{defn} \cite{DP}
    Let $\cX$ be some domain (for instance, $\cX$ could be $\BR^d$), and let $0 \le \eps, \delta \le 1$ be parameters. A randomized algorithm $\cA: \cX^n \to \cO$ that takes in $n$ datapoints $X_1, X_2, \dots, X_n \in \cX$ and outputs some $o \in \cO$ is $(\eps, \delta)$-\emph{differentially private} ($(\eps, \delta)$-DP) if for any two datasets $\textbf{X} = (X_1, \dots, X_i, \dots X_n)$ and $\textbf{X}' = (X_1, \dots, X_i', \dots, X_n)$ that differ in only a single data point, and any subset $O \subset \cO$,
\[\BP\left[\cA(\textbf{X}') \in O\right] \le e^{\eps} \cdot \BP\left[\cA(\textbf{X}) \in O\right] + \delta.\]
\end{defn}
Intuitively, if an external adversary sees the output of an $(\eps, \delta)$-DP algorithm $\cA$, then with at most $\delta$ failure probability, it is impossible to distinguish between the $i^{\text{th}}$ data point being either some $X_i$ or some other $X_i'$ with more than an $\eps$ advantage, even if the adversary had unbounded computational power. Hence, differential privacy is an information theoretic way of masking any individual data point, and keeping each data point safe from adversaries. The $\delta$ additive error can sometimes result in a complete leakage of a data point (for instance, an algorithm that outputs $X_1$ with $\delta$ probability and nothing otherwise is $(0, \delta)$-DP). Therefore, in differential privacy, one wishes for $\delta$ to be very small: usually one wishes for $\delta= n^{- \omega(1)}$, i.e., $\delta$ decays \emph{super-polynomially} with the dataset size. In fact, one may even wish for $\delta = 0$: this is often called \emph{pure differential privacy} (pure-DP), as opposed to \emph{approximate differential privacy} (approximate-DP) when $\delta > 0$.

\medskip

From the perspective of differential privacy, algorithmic statistics has enjoyed a significant amount of work over the past several years, with numerous papers studying differentially private mean~\cite{KarwaV18,KamathLSU19,BunKSW19,KamathSU20,AdenAliAK21,LiuKKO21,BrownGSUZ21,LevySAKKMS21,HuangLY21,HopkinsKM22,NarayananME22,TsfadiaCKMS22,CummingsFMT22,HopkinsKMN23,DuchiHK23,BrownHS23} and covariance~\cite{KarwaV18,AminDKMV19,KamathLSU19,BunKSW19,AdenAliAK21,LiuKO22,KamathMSSU22,KothariMV22,AshtianiL22,TsfadiaCKMS22,KamathMS22,DongLY22,HopkinsKMN23,AlabiKTVZ23} estimation in high dimensions. Much of this work has focused on the setting where the samples are drawn i.i.d. from a Gaussian distribution. This has led to optimal sample complexity bounds for estimating both identity-covariance Gaussians and arbitrary Gaussians~\cite{KamathLSU19,AdenAliAK21,KamathMS22} in total variation distance, as well as matching polynomial-time algorithms~\cite{KamathLSU19,AshtianiL22,HopkinsKMN23}. Recently, there has also been work on private ``covariance-aware mean estimation'', where one wishes to estimate the mean of an unknown-covariance Gaussian: for this problem, we have optimal sample complexity bounds~\cite{BrownGSUZ21} and nearly matching efficient algorithms~\cite{DuchiHK23, BrownHS23}. Other problems that have been studied include private mean estimation for heavy-tailed distributions~\cite{KamathLSU19,HopkinsKM22} and private mean/covariance estimation for arbitrary bounded data~\cite{AminDKMV19, LevySAKKMS21, HuangLY21, NarayananME22, DongLY22}. 
In addition to being an extremely fundamental problem, private mean estimation has proven to be a valuable subroutine in numerous other private algorithms, most notably in optimization tasks requiring private stochastic gradient descent (e.g.,~\cite{AbadiCGMMTZ16, BassilyFTT19}).

\medskip

Despite the large body of work on mean and covariance estimation, we still do not have a full understanding of these problems. One such problem is heavy-tailed mean estimation with bounded $k^{\text{th}}$ moments. Namely, we are promised that for some fixed constant $k \ge 2$, the (high-dimensional) data comes from a distribution $\mathcal{D}$ with unknown mean $\mu$, but with bounded $k^{\text{th}}$ moment around $\mu$ in every direction, i.e., $\BE_{X \sim \mathcal{D}} \left|\langle X-\mu, v \rangle\right|^k \le O(1)$ for every unit vector $v \in \BR^d$. We wish to privately learn $\hat{\mu}$ such that $\|\hat{\mu}-\mu\|_2 \le \alpha.$
The second is that while we understand the complexity of private Gaussian covariance estimation up to small Frobenius error (which corresponds to the notation of total variation distance), we do not yet understand the complexity of estimation up to spectral error. In Frobenius error, given samples from $\cN(\mu, \Sigma)$, we wish to privately learn some $\hat{\Sigma}$ such that $\|\Sigma^{-1/2} \hat{\Sigma} \Sigma^{-1/2} - I\|_F \le \alpha,$ whereas in spectral error, we wish to privately learn $\hat{\Sigma}$ such that $\|\Sigma^{-1/2} \hat{\Sigma} \Sigma^{-1/2} - I\|_{op} \le \alpha,$ or equivalently, $(1-\alpha) \Sigma \preccurlyeq \hat{\Sigma} \preccurlyeq (1+\alpha) \Sigma$, where $\preccurlyeq$ represents the Loewner ordering.

\subsection{This work}

In this work, we prove optimal lower bounds for both private heavy-tailed mean estimation and private Gaussian covariance estimation in spectral error, matching known upper bounds.
Our lower bounds are against approximate-DP algorithms (and thus automatically hold against pure-DP algorithms).
We now state our lower bounds, starting with our result on covariance estimation in spectral error.

\begin{theorem}[Informal, see \Cref{thm:covariance-formal}] \label{thm:covariance}
    For any $\alpha, \eps \le O(1),$ and any $\delta \le (\frac{\alpha \cdot \eps}{d})^{O(1)}$, any $(\eps, \delta)$-DP algorithm that solves covariance estimation up to spectral error $\alpha$ for Gaussians in $d$ dimensions requires sample complexity
\[n \ge \tilde{\Omega}\bigg(\underbrace{\frac{d}{\alpha^2}}_{\clap{\scriptsize required even for non-private algorithms}} + \frac{d^{3/2}}{\alpha \eps}\bigg).\]
\end{theorem}

\Cref{thm:covariance} improves over the best-known lower bound of~\cite{KamathMS22}, which had a matching sample complexity bound but only held for $\alpha \le O(\frac{1}{\sqrt{d}})$.
Moreover, \Cref{thm:covariance} matches known algorithms of~\cite{KamathLSU19, BrownHS23}, up to logarithmic factors in $d, \frac{1}{\alpha}, \frac{1}{\eps}, \frac{1}{\delta}$. Hence, up to logarithmic factors, this completes the picture for private Gaussian covariance estimation up to spectral error. We also remark that our proof generalizes Theorem 1.1 in~\cite{KamathMS22} while also being simpler.

Next, we state our lower bound for heavy-tailed mean estimation.
\begin{theorem}[Informal, see~\Cref{thm:heavy-tailed-formal}] \label{thm:heavy-tailed}
    For any $\alpha, \eps \le O(1),$ and any $\delta \le (\frac{\alpha \cdot \eps}{d})^{O(1)}$, any $(\eps, \delta)$-DP algorithm that solves mean estimation up to error $\alpha$ for heavy-tailed distributions with bounded $k^{\text{th}}$ moment in $d$ dimensions requires sample complexity
\[n \ge \tilde{\Omega}\bigg(\underbrace{\frac{d}{\alpha^2}}_{\clap{\scriptsize required even for non-private algorithms}} + \frac{d}{\alpha^{k/(k-1)} \eps}\bigg).\]
\end{theorem}

Theorem~\ref{thm:heavy-tailed} improves on the best-known lower bound, which had a matching sample complexity bound but only held for pure-DP algorithms~\cite{BarberD14, KamathSU20}. As pure-DP is more stringent than approximate-DP, it is more difficult to prove approximate-DP lower bounds: a matching approximate-DP lower bound is only known for Gaussian distributions or when $k=2$~\cite{KamathLSU19, KamathMS22}. Moreover, \Cref{thm:heavy-tailed} matches a known algorithm (upper bound) of~\cite{KamathSU20}, up to logarithmic factors in $d, \frac{1}{\alpha}, \frac{1}{\eps}, \frac{1}{\delta}$. Hence, up to logarithmic factors, this essentially completes the picture for private heavy-tailed mean estimation.

\paragraph{Implications:}
Theorem~\ref{thm:covariance} leads to two important implications. The first is a dimension-based separation between the sample complexity of robustness and privacy. Specifically, it is known that \emph{robustly} learning the covariance of a Gaussian up to spectral error $\alpha$ only requires $O(\frac{d}{\alpha^2})$ samples, though all known algorithms that run in polynomial time use $\Omega(d^2)$ samples (e.g., see~\cite[Section 6]{DiakonikolasKS17}). Hence, the sample complexity of approximate-DP spectral covariance estimation has a greater polynomial dependence on the dimension ($d^{3/2}$) than the sample complexity of robust spectral covariance estimation ($d$). To our knowledge, this is the first such \emph{dimension-based} separation known for a statistical estimation problem, where the sample complexity for robustness (ignoring runtime constraints) is \emph{strictly smaller} than the sample complexity for approximate differential privacy.

Second, Theorem~\ref{thm:covariance} leads to improved lower bounds for private empirical covariance estimation of arbitrary bounded data. Given data points $X = \{X_1, \dots, X_n\}$ that are promised to lie in a $d$-dimensional ball of radius $1$, one can privately estimate the empirical covariance $\hat{\Sigma}$ of $X$, up to error $\|\hat{\Sigma}-\Sigma\|_F \lesssim \min\left(\frac{d}{n}, \frac{d^{1/4}}{\sqrt{n}}\right),$ ignoring polynomial factors in $\eps, \log \frac{1}{\delta}$~\cite{NikolovTZ13,DworkNT15,DongLY22}. The best corresponding lower bound for the Frobenius error is a matching $\Omega\left(\frac{d}{n}\right)$ when $d \le \sqrt{n}$~\cite{KasiviswanathanRSU10}, but is only $\frac{1}{\sqrt{n}}$ for $\sqrt{n} \le d \le n$~\cite{KasiviswanathanRSU10} and $\frac{\sqrt{d}}{n}$ for $n \le d \le n^2$~\cite{KamathLSU19}  (see \cite[Figure 1]{DongLY22} for a depiction of both the upper and lower bounds). Our proof of Theorem~\ref{thm:covariance} can be used to improve the lower bound to a tight $\Omega\left(\frac{d}{n}\right)$ for $\sqrt{n} \le d \le n^{2/3}$, and an improved $\Omega(\frac{1}{n^{1/3}})$ for $n^{2/3} \le d \le n^{4/3}$.
The upper and lower bounds still do not match when $n^{2/3} \le d \le n^2$, and a natural follow-up question is to close this gap.

\subsection{Additional related work}

Our techniques are based on the technique of \emph{fingerprinting lower bounds} for differential privacy, which was first used in a work of Bun et al.~\cite{BunUV14}. Since then, there have been various other privacy lower bounds based on the fingerprinting technique~\cite{HardtU14,SteinkeU15,DworkSSUV15,SteinkeU16,SteinkeU17,BunSU19,KamathLSU19,CaiWZ20,NarayananME22,KamathMS22,CaiWZ23,PeterTU23,survey}. While fingerprinting lower bounds are mainly used in the approximate-DP setting, it is more common to use \emph{packing lower bounds} in the pure-DP setting (see, e.g.,~\cite{HardtT10}).

The most relevant works to ours are perhaps those of~\cite{KamathLSU19, KamathMS22}, which prove lower bounds for mean and covariance estimation of Gaussians with approximate DP. The works of \cite{KamathMS22, CaiWZ23} give a technique for lower bounds for general exponential families, although \cite{CaiWZ23} specifically considers other statistical problems and does not consider mean or covariance estimation of distributions. The main ``score attack'' statistic they use in their lower bound is also different from ours.

\subsection{Roadmap}

In \Cref{sec:overview}, we give a technical overview of the proofs of Theorems~\ref{thm:covariance} and~\ref{thm:heavy-tailed}. In \Cref{sec:prelim}, we note some useful facts and concentration bounds. In \Cref{sec:covariance}, we prove \Cref{thm:covariance}. In \Cref{sec:heavy-tailed}, we prove \Cref{thm:heavy-tailed}. Finally, in \Cref{sec:empirical}, we explain how our proof of \Cref{thm:covariance} implies an improved lower bound for empirical covariance estimation.

\section{Proof Overview} \label{sec:overview}

\paragraph{Fingerprinting overview:} We first describe a general approach explaining fingerprinting lower bounds. This approach mirrors the other fingerprinting lower bounds in private statistical estimation~\cite{KamathLSU19, survey, KamathMS22}.

Suppose we are trying to estimate a parameter $\theta$ that characterizes a distribution $\mathcal{D}_\theta$. (For covariance estimation, $\theta = \Sigma$ and $\mathcal{D}_\theta = \cN(0, \Sigma)$.)
We fix a $(\eps, \delta)$-DP mechanism $M$ with input $X_1, \dots, X_n \sim \mathcal{D}_\theta$ and with output some estimate $\hat{\theta}$. Consider drawing i.i.d. samples $X_1, \dots, X_n \sim \mathcal{D}_\theta$ and fresh i.i.d. samples $X_1', \dots, X_n'\sim \mathcal{D}_\theta$, and for each index $i \in [n]$, define the statistics
\begin{equation} \label{eq:main-statistic}
    Z_i := \langle f(X_1, \dots, X_i, \dots, X_n, \theta), g(X_i, \theta) \rangle \hspace{0.5cm} \text{and} \hspace{0.5cm} Z_i' := \langle f(X_1, \dots, X_i', \dots, X_n, \theta), g(X_i, \theta) \rangle,
\end{equation}
for some fixed functions $f, g$, where $f$ will depend only on $M(X_1, \dots, X_n)$ and $\theta$.
The idea is that $Z_i'$ is the inner product of two independent quantities (since $X_i$ is not in the set $\{X_1, \dots, X_i', \dots, X_n\}$), which makes it easier to bound the mean and variance of $Z_i'$. Moreover, if $M$ is a private algorithm, then the distribution of $Z_i$ and $Z_i'$, even for \emph{fixed} samples $\{X_i\}, \{X_i'\}$ and $\theta$, are close, which means the overall distribution of $Z_i$ and $Z_i'$ are similar after removing the conditioning on $\{X_i\}, \{X_i'\}$ and $\theta$. Hence, we can also bound the distribution of $Z_i$, and thus bound $\BE[Z_i]$.

Conversely, we will show that if $M$ is a reasonably accurate estimator, then $\BE\left[\sum_{i=1}^n Z_i\right]$ will have to be large compared to our bounds on each $\BE[Z_i]$, unless $n$ is sufficiently large. To actually prove this, we first carefully choose the functions $f$ and $g$ as well as the prior distribution on the parameter $\theta$. Then, we prove a ``fingerprinting'' lemma, which proves if $X_1, \dots, X_n \sim \mathcal{D}_\theta$, then either $M(X_1, \dots, X_n)$ is not a good estimate for $\theta$ with reasonable probability, or $\BE\left[\sum_{i=1}^n Z_i\right]$ is large. The main technical difficulties lie in choosing the functions and distributions, and then proving the fingerprinting lemma.

\paragraph{Spectral covariance estimation:} We will prove a stronger statement: namely, for any $\alpha \le O(\sqrt{d})$, there exists a distribution $\mathcal{P}$ on the covariance $\Sigma$ with the following two properties.
\begin{enumerate}
    \item With very high probability, $\Sigma \sim \mathcal{P}$ has all eigenvalues $\Theta(1)$.
    \item For any $(\eps, \delta)$-DP algorithm $M(X_1, \dots, X_n),$ if $\Sigma \sim \mathcal{P}$ and $X_1, \dots, X_n \sim \cN(0, \Sigma)$, where $\BE[\|M(X_1, \dots, X_n)-\Sigma\|_F^2] \le \alpha^2$, then we must have $n \ge \Omega\big(\frac{d^2}{\alpha \eps}\big)$.
\end{enumerate}

By setting $\alpha' = \frac{\alpha}{\sqrt{d}}$ (so $\frac{d^2}{\alpha \eps} 
= \frac{d^{3/2}}{\alpha' \eps}$), this implies that we cannot have $\|M(X_1, \dots, X_n)-\Sigma\|_{op} \le \alpha'$ with very high probability, and since all eigenvalues of $\Sigma$ are $\Theta(1)$, this implies our desired result. This holds for any $\alpha \le O(\sqrt{d})$, and hence for any $\alpha' \le O(1)$.

The choices of $f, g$ in \eqref{eq:main-statistic} will be quite simple: we choose $f(X_1, \dots, X_n, \Sigma) = M(X_1, \dots, X_n)-\Sigma$ and $g(X_i) = X_i X_i^\top - \Sigma$, so
\[Z_i := \langle M(X_1, \dots, X_n)-\Sigma, X_i X_i^\top - \Sigma \rangle \hspace{0.5cm} \text{and} \hspace{0.5cm} Z_i' =  \langle M(X_1, \dots, X_i', \dots, X_n)-\Sigma, X_i X_i^\top - \Sigma\rangle.\]
Using the fact that $X_i X_i^\top$ is an unbiased estimator for $\Sigma$, a simple calculation shows that $\BE[Z_i'] = 0$. Moreover, assuming $M$ is a reasonably good estimator of $\Sigma$, we can show $\Var(Z_i') \le O(\alpha^2)$. Given these, $(\eps, \delta)$-DP will imply for reasonably small $\delta$ that $\BE[Z_i] \le O(\alpha \eps)$ for all $i$. Hence, $\BE[\sum_{i=1}^n Z_i] \le O(n \cdot \alpha \eps)$ if $M$ is differentially private and reasonably accurate.

Next, we show a lower bound on $\BE[\sum_{i=1}^n Z_i]$, assuming $M$ is a sufficiently accurate estimator. This lower bound does not utilize any privacy constraints. Note that $\BE[\sum_{i=1}^n Z_i] = n \cdot \BE[\langle M(X_1, \dots, X_n)-\Sigma, \hat{\Sigma}-\Sigma \rangle]$, where $\hat{\Sigma} = \frac{1}{n} \sum_{i=1}^n X_i X_i^\top$ is the empirical covariance. So, we want to show this quantity is larger than $O(n \cdot \alpha \eps)$, which contradicts our above bound, unless either $n \ge \Omega\left(\frac{d^2}{\alpha \eps}\right)$ or $\|M(X_1, \dots, X_n)-\Sigma\|_F > \alpha$ holds with reasonable probability.

We can rewrite our desired quantity as 
\begin{align} \label{eq:sum_zi}
    \BE\left[\sum_{i=1}^n Z_i\right] = n \cdot \left(\BE \left\langle M(X_1, \dots, X_n)-\hat{\Sigma}, \hat{\Sigma}-\Sigma\right\rangle + \BE\left[\|\hat{\Sigma}-\Sigma\|_F^2\right]\right).
\end{align}
It is well-known that $\BE\left[\|\hat{\Sigma}-\Sigma\|_F^2\right] = \Theta(\frac{d^2}{n})$. Also, we can write
\begin{align}
    \left|\BE \left\langle M(X_1, \dots, X_n)-\hat{\Sigma}, \hat{\Sigma}-\Sigma\right\rangle\right| 
    &= \left|\BE \left\langle M(X_1, \dots, X_n)-\hat{\Sigma}, \hat{\Sigma}-\BE[\Sigma|X_1, \dots, X_n]\right\rangle\right| \nonumber \\
    &\le \sqrt{\BE\|M(X_1, \dots, X_n)-\hat{\Sigma}\|_F^2 \cdot \BE\|\hat{\Sigma} - \BE[\Sigma|X_1, \dots, X_n]\|_F^2}. \label{eq:fingerprint}
\end{align}
Above, we can replace $\Sigma$ with the conditional expectation $\BE[\Sigma|X_1, \dots, X_n]$, because the left-hand side of the inner product only depends on $X_1, \dots, X_n$.

Assuming that $M(X_1, \dots, X_n)$ is a good estimator of $\Sigma$, it will also be a good estimator of $\hat{\Sigma},$ and $\BE\|M(X_1, \dots, X_n)-\hat{\Sigma}\|_F^2 \le \alpha^2$. 
We have avoided discussing the prior distribution of $\Sigma$, but to bound $\BE\|\hat{\Sigma} - \BE[\Sigma|X_1, \dots, X_n]\|_F^2,$ we need to define the prior. The prior that we choose will be an \emph{Inverse Wishart} distribution, which is known to be the classic \emph{conjugate prior} of the Multivariate Gaussian. What this means is that if the prior distribution of $\Sigma$ follows an Inverse Wishart distribution and we sample $X_1, \dots, X_n \sim \mathcal{N}(0, \Sigma)$, the posterior distribution of $\Sigma$ given $X_1, \dots, X_n$ also follows an Inverse Wishart distribution (with a different parameter setting). This will make it easy to compute $\BE[\Sigma|X_1, \dots, X_n]$. We will choose $\Sigma$ to be a (scaled) Inverse Wishart distribution with $C \cdot d$ degrees of freedom for a sufficiently large constant $C$. With a proper scaling, all of the eigenvalues of $\Sigma$ will be between $0.5$ and $1.5$, and the posterior distribution will have expectation $\left(1+O(\frac{d}{n})\right) \cdot \hat{\Sigma}$. From this, it is not hard to bound $\BE\|\hat{\Sigma} - \BE[\Sigma|X_1, \dots, X_n]\|_F^2 \le O(\frac{d^2}{n^2}) \cdot \BE[\|\hat{\Sigma}\|_F^2] = O\left(\frac{d^3}{n^2}\right)$. Combining this with Equations \eqref{eq:sum_zi} and \eqref{eq:fingerprint} and our bound $\BE\|M(X_1, \dots, X_n)-\hat{\Sigma}\|_F^2 \le \alpha^2$, this implies that
\[\BE\left[\sum_{i=1}^n Z_i \right] = n \cdot \left[\Theta\left(\frac{d^2}{n}\right) \pm O\left(\sqrt{\alpha^2 \cdot \frac{d^3}{n^2}}\right)\right].\]
As long as $\alpha \le c \sqrt{d}$ for some small constant $c$, this implies $\BE[\sum_{i=1}^n Z_i] \ge \Omega(d^2)$. As we already explained why $\BE[\sum_{i=1}^n Z_i] \le O(n \cdot \alpha \eps)$, this implies that as long as $\alpha \le c \sqrt{d}$, any $(\eps, \delta)$-DP algorithm that can estimate $\Sigma$ up to Frobenius error $\alpha$ needs $O(n \cdot \alpha \eps) \ge \Omega(d^2),$ or $n \ge \Omega\left(\frac{d^2}{\alpha \eps}\right).$

\paragraph{Heavy-tailed mean estimation:} This result will follow from a simple application of the fact that privately learning $\mu$ up to error $\alpha$ requires $\Omega(\frac{d}{\alpha \eps})$ samples from $\cN(\mu, I)$~\cite{KamathLSU19}. Specifically, we will draw a distribution that, with probability $\alpha^{k/(k-1)}$ is drawn as $\cN(\mu', \alpha^{-2/(k-1)} \cdot I)$ for some unknown $\mu'$ with $\|\mu'\| \le O(\alpha^{-1/(k-1)})$. It is straightforward to check that this distribution has bounded $k$th moment, and the actual mean, $\mu = \alpha^{k/(k-1)} \cdot \mu'$, has norm $O(\alpha)$. However, to learn $\mu$ to error $\alpha$, one must learn $\mu'$ up to error exactly $\alpha^{-1/(k-1)}$, not just $O(\alpha^{-1/(k-1)})$. A minor modification of the lower bound in~\cite{KamathLSU19} can show that learning $\mu'$ is essentially equivalent to learning the mean of identity covariance Gaussian up to error $1$. This requires $\Omega\left(\frac{d}{\eps}\right)$ samples. However, because only an $\alpha^{k/(k-1)}$ fraction of the points were actually from the Gaussian, we need $\Omega\left(\frac{d}{\eps \cdot \alpha^{k/(k-1)}}\right)$ samples in total.
This argument can be made formal by converting an instance of Gaussian estimation into this distribution by padding.

\section{Preliminaries} \label{sec:prelim}

\paragraph{Statistical estimation:} We will need a few known results about statistical estimation.

First, we note a well-known bound regarding the accuracy of the empirical covariance matrix.

\begin{lemma}[Folklore] \label{lem:emp-cov-acc}
    For any fixed $\Sigma$, suppose $X_1, \dots, X_n$ are drawn from $\cN(0, \Sigma)$, and let $\hat{\Sigma}=\frac{1}{n} \sum_{i=1}^n X_iX_i^\top$. Then, $2 \cdot \Tr(\Sigma)^2/n \le \BE[\|\hat{\Sigma}-\Sigma\|_F^2] = 3 \cdot \Tr(\Sigma)^2/n$. Importantly, this means $2 \cdot \lambda_{\min}(\Sigma)^2 \cdot \frac{d^2}{n} \le \BE[\|\hat{\Sigma}-\Sigma\|_F^2] \le 3 \cdot \lambda_{\max}(\Sigma)^2 \cdot \frac{d^2}{n}$.
\end{lemma}

Next, we note the known upper and lower bounds for private Gaussian mean estimation.

\begin{theorem}~\cite{KamathLSU19, AdenAliAK21, HopkinsKMN23} \label{thm:mean}
    Fix parameters $\eps, \delta, \alpha \le 1$. Then, one can learn the mean $\mu$ of an identity-covariance Gaussian with $(\eps, \delta)$-DP, up to error $\alpha$, with $\tilde{O}\left(\frac{d}{\alpha^2}+\frac{d}{\alpha \eps} + \frac{\log (1/\delta)}{\eps}\right)$ samples.

    Moreover, if $\delta \le \left(\frac{\alpha \eps}{d}\right)^{C_1}$ for some fixed constant $C_1=O(1)$, learning $\mu$ up to error $\alpha$ with $(\eps, \delta)$-DP requires $\tilde{\Omega}\left(\frac{d}{\alpha^2} + \frac{d}{\alpha \eps}\right)$ samples.
\end{theorem}

\paragraph{Wishart distributions:}
We start by introducing the \emph{Wishart} and \emph{Inverse Wishart} distributions, and some useful facts about them.

\begin{defn}[Wishart Distribution]
    The $d$-dimensional \emph{Wishart} distribution with $m$ degrees of freedom and scale $V$ (where $V$ is a $d \times d$ symmetric and positive definite matrix), written as $W_d(V, m)$, is a distribution over $d \times d$-dimensional matrices generated as follows. We sample $m$ $d$-dimensional Gaussians $g_1, \dots, g_m \overset{i.i.d.}{\sim} \cN(0, V)$, and let $W_d(V, m)$ be $\sum_{i=1}^m g_i g_i^\top$.
\end{defn}

\begin{defn}[Inverse Wishart Distribution]
    The $d$-dimensional \emph{Inverse Wishart} distribution with $m$ degrees of freedom and scale $V$ (where $V$ is a $d \times d$ symmetric and positive definite matrix), written as $W_d^{-1}(V, m)$, has distribution where we generate $W \sim W_d(V^{-1}, m)$ and output its inverse.
\end{defn}

    It is well-known that if $m > d$, $W_d(V, m)$ and $W_d^{-1}(V, m)$ are symmetric and positive definite with probability $1$. It is well known that these distributions have the following probability density functions. 

\begin{proposition} \label{prop:wishart-pdf}
    Suppose that $m \ge d+2$.
    At any symmetric positive definite $\Sigma$, the PDF of the Wishart distribution $W_d(V, m)$ at $\Sigma$ is proportional to $(\det \Sigma)^{(m-d-1)/2} \cdot e^{-\Tr(V^{-1} \cdot \Sigma)/2}$. The PDF of the Inverse-Wishart distribution $W_d^{-1}(V, m)$ at $\Sigma$ is proportional to $(\det \Sigma)^{-(m+d+1)/2} \cdot e^{-\Tr(V \cdot \Sigma^{-1})/2}$.

    Here, we omit normalizing factors that only depend on $V, m,$ and $d$ (but are independent of $\Sigma$).
\end{proposition}

    The expecations of the Wishart and Inverse-Wishart matrices are also well-characterized.

\begin{proposition} \label{prop:wishart-expectation}
    The expectation of $W_d(V, m)$ is $m \cdot V$. For $m \ge d+2$, the expectation of $W_d^{-1}(V, m)$ is $\frac{V}{m-d-1}$.
\end{proposition}

    We will also need some tail bounds for Inverse-Wishart matrices.

\begin{lemma} \label{lem:wishart-operator-norm-bound}
    Fix any even integer $k \ge 2$. Then, if $m \ge 2 \cdot d$ and $M \sim W_d^{-1}(I, m)$, then $\BP(\lambda_{\max}(M) \ge x/m) \le (e^2/x)^{d/2}$ for any positive real value $x > 0$. This implies that, if $d \ge 10 k$, then $\BE[\lambda_{\max}(M)^k] \le G_k/m^k$ for some constant $G_k$ only depending on $k$.

    Moreover, with at least $2/3$ probability, $\lambda_{\min}(M)$ is at least $\Omega(\frac{1}{m})$. Thus, there is a constant $g_k$ such that $\BE[\lambda_{\min}(M)^k] \ge g_k/m^k$.
\end{lemma}

The proof is essentially immediate from known results~\cite{ChenD05, VershyninBook}, so we defer the proof to Appendix~\ref{sec:prelim-app}.


\paragraph{Concentration bounds:} Here, we will state some more simple but useful concentration bounds.

First, we need the following bound.

\begin{proposition} \label{prop:4th-moment}
    Let $P$ be a $d \times d$ matrix, and $X \sim \cN(0, \Sigma)$, where $\|\Sigma\|_{op}$. Then,
\[\BE\left[\langle P, X X^\top - \Sigma \rangle^2\right] \le 2 \cdot \|\Sigma\|_{op}^2 \cdot \|P\|_F^2.\]
\end{proposition}

Next, we recall the Hanson-Wright inequality.

\begin{lemma} \label{lem:hanson-wright} [Hanson-Wright Inequality]
    Let $X \sim \cN(0, \Sigma)$ be a $d$-dimensional Standard Gaussian. Then, there exists an absolute constant $c_1$ such that for any $d \times d$ symmetric matrix $A$,
\[\BP\left(|\|X\|^2 - \Tr(\Sigma)| \ge t\right) \le 2 \exp\left(-c_1 \cdot \min\left(\frac{t^2}{\|\Sigma\|_F^2}, \frac{t}{\|\Sigma\|_{op}}\right)\right).\]
\end{lemma}

Next, we need a simple bound on the norms of Gaussians.

\begin{proposition} \label{prop:conc_bound}
    Suppose that $X \sim \cN(0, \Sigma)$. Then, for any fixed integer $k \ge 1$, $\BE[\|X\|^{2k}] \le H_k \cdot \|\Sigma\|_{op}^k \cdot d^k,$ where $H_k$ is a constant only depending on $k$.
\end{proposition}

Finally, we need the following proposition comparing the expectations of similar random variables (where similarity corresponds to the notion of approximatel differential privacy).

\begin{proposition} \label{prop:dumb_calculation}
    Suppose $0 \le \eps \le 1$ and $0 \le \delta \le 1/2$, and that $X, Y$ are real-valued random variables such that for any set $S$, $e^{-\eps} \cdot \BP(X \in S) - \delta \le \BP(Y \in S) \le e^{\eps} \cdot \BP(X \in S) + \delta$. Then, $\left|\BE[X-Y]\right| \le 2 \eps \cdot \BE[|X|] + 2 \sqrt{\delta \cdot \BE[X^2+Y^2]}$.
\end{proposition}

The proofs of Propositions~\ref{prop:4th-moment}, \ref{prop:conc_bound}, and \ref{prop:dumb_calculation} are standard and are deferred to \Cref{sec:prelim-app}.

\section{Lower Bound for Private Gaussian Covariance Estimation} \label{sec:covariance}

\subsection{Setup and assumptions} \label{sec:setup}

Our goal will be to prove the following theorem, which formalizes \Cref{thm:covariance}.

\begin{theorem} \label{thm:covariance-formal}
    Let $c_4 < 1$ be a sufficiently small absolute constant. Suppose that $\alpha \le c_4$, $\eps \le 1$, and $\delta \le \frac{\eps^2}{d^2}$.
    Suppose that $M$ is an $(\eps, \delta)$-DP mechanism that takes as input $X = \{X_1, \dots, X_n\}$ where each $X_i \in \BR^d$. Suppose that for any positive definite matrix $\Sigma$, if $X = \{X_1, \dots, X_n\} \overset{i.i.d.}{\sim} \cN(0, \Sigma)$, then with probability at least $2/3$ over the randomness of $M$ and $X_1, \dots, X_n$, $(1-\alpha) \Sigma \preccurlyeq M(X) \preccurlyeq (1+\alpha) \Sigma$. Then, we must have that $n \ge \tilde{\Omega}\left(\frac{d}{\alpha^2} + \frac{d^{3/2}}{\alpha \eps}\right)$.
\end{theorem}

We consider sampling $\Sigma$ from some prior distribution $p_0$ - we will choose an Inverse Wishart prior distribution.
Specifically, we set $p_0 \sim W_d^{-1}(\frac{I}{m}, m)$, where $m = 2d$. By Lemma~\ref{lem:wishart-operator-norm-bound} and a simple scaling, we have 
\begin{equation} \label{eq:prior-bound-1}
    \BP_{\Sigma \sim p_0}(\|\Sigma\|_{op} \ge x) \le (e^2/x)^{d/2},
\end{equation}
and if $d$ is sufficiently large, by \Cref{lem:wishart-operator-norm-bound},
\begin{equation} \label{eq:prior-bound-2}
    \BE[\|\Sigma\|_{op}^4] \le G_4 = O(1).
\end{equation}


Fix an $(\eps, \delta)$-DP mechanism $M: (\BR^d)^n \to \BR^{d \times d}$, that takes as input a dataset $X = \{X_1, \dots, X_n\}$ of i.i.d. samples from $\cN(0, \Sigma).$
Let $\gamma = 10 \sqrt{d} \cdot \alpha$, and suppose we have a constant-probability bound on the Frobenius error for all $\Sigma$ with spectral norm at most $10$:
\begin{equation} \label{eq:assumption}
    \BP_{X, M} \left(\|M(X) - \Sigma\|_F \le \gamma \right) \ge 2/3, \hspace{0.5cm} \text{if} \hspace{0.5cm} \|\Sigma\|_{op} \le 10.
\end{equation}
Note that this is a weaker assumption than in \Cref{thm:covariance-formal}, since $(1-\alpha)\Sigma \preccurlyeq M(X) \preccurlyeq (1+\alpha) \Sigma$ and $\|\Sigma\|_{op} \le 10$ implies that $\|M(X)-\Sigma\|_F \le \sqrt{d} \cdot \|M(X)-\Sigma\|_{op} \le \sqrt{d} \cdot \alpha \cdot \|\Sigma\|_{op} \le \gamma$.
Then, if $\Sigma$ is drawn from the inverse Wishart prior as above, we can strengthen this to assuming
\begin{equation} \label{eq:assumption-stronger}
    \BE_{\Sigma, X, M}(\|M(X)-\Sigma\|_F^4) \le \gamma^4,
\end{equation}
at the cost of requiring $O\left(\log \frac{d}{\gamma}\right)$ times as many samples. We defer the proof of this to \Cref{subsec:assumptions}. We will show that under the stronger assumption \eqref{eq:assumption-stronger}, $M$ requires $\Omega\left(\frac{d^2}{\gamma^2}+\frac{d^2}{\gamma \eps}\right)$ samples, so under \eqref{eq:assumption}, $M$ requires $\Omega\left(\frac{d^2}{\gamma^2 \cdot \log(d/\gamma)}+\frac{d^2}{\gamma \eps \cdot \log(d/\gamma)}\right)$ samples.
This reduction will assume that $\gamma \le c \sqrt{d}$ and $\gamma \ge e^{-c d}$ for some small constant $c$, and that $d$ is a sufficiently large constant.

\medskip


Moreover, we can remove the assumption that $d$ is at least a sufficiently large constant, and that $\gamma \ge e^{-\Omega(d)}$. Indeed, in the case of $d = 1$, there is a known $\Omega\left(\frac{1}{\alpha^2}+\frac{1}{\alpha \eps}\right)$ lower bound~\cite{KarwaV18, KamathLSU19} (which trivially extends to higher dimension $d$), and if $\alpha^{-1} \ge \gamma^{-1} \ge e^{\Omega(d)}$ then the factor of $d$ can be absorbed as a $\log(1/\alpha)$ factor.

\medskip

Hence, our goal in the remainder of the section is to prove the following theorem.

\begin{theorem} \label{thm:main}
    There exist constants $c_2, c_3 < 1 < C_4$ with the following properties. Suppose $d \ge C_4$, $m = 2d$, $\gamma \le c_2 \sqrt{d}, \eps \le 1,$ and $\delta \le \frac{\eps^2}{d^2}$.
    Let $M$ be an $(\eps, \delta)$-DP mechanism that takes as input $X = \{X_1, \dots, X_n\}$ where each $X_i \in \BR^d$. Suppose that if $X_1, \dots, X_n \overset{i.i.d.}{\sim} \cN(0, \Sigma)$, where $\Sigma \sim W_d^{-1}(\frac{I}{m}, m)$, then $\BE_{\Sigma, X, M}(\|M(X)-\Sigma\|_F^4) \le \gamma^4$. Then, we must have that $n \ge c_3 \cdot \max\left(\frac{d^2}{\gamma^2}, \frac{d^2}{\gamma \eps}\right)$.
\end{theorem}

Indeed, by replacing $\gamma$ with $10 \sqrt{d} \cdot \alpha$, and dividing the sample complexity by $\log(d/\gamma) = \Theta(\log(\sqrt{d}/\alpha))$ to replace the assumption~\eqref{eq:assumption-stronger} with \eqref{eq:assumption}, we see that \Cref{thm:main} implies \Cref{thm:covariance-formal}.

\paragraph{Fingerprinting statistics:}
As in other private estimation lower bounds (see, e.g., \cite{survey}), we also consider drawing additional i.i.d. samples $X' = \{X_1', \dots, X_n'\}$ from $\cN(0, \Sigma).$
For each $i \in [n]$, define
\begin{align*}
    Z_i &:= \langle M(X)-\Sigma, X_i X_i^\top - \Sigma \rangle, \\
    Z_i' &:= \langle M(X_{\sim i})-\Sigma, X_i X_i^\top - \Sigma \rangle.
\end{align*}
Here, $X_{\sim i}$ is the dataset where we replace $X_i$ with $X_i'$, so it differs in exactly one location from $X$.

In the next two subsections, we prove an upper and lower bound, respectively, on the quantity $\BE[\sum_{i=1}^n Z_i]$, which will form a contradiction unless $n \ge \Omega\left(\frac{d^2}{\gamma^2}+\frac{d^2}{\gamma \eps}\right)$.

\subsection{Upper bound}

In this subsection, we provide an upper bound for $\BE_{\Sigma, X, M}[\sum_{i=1}^n Z_i]$.
First, we note the following.

\begin{proposition} \label{prop:Zi_prime_bound}
    Assume the conditions of \Cref{thm:main}. Then, for every $i \in [n]$, $\BE[Z_i'] = 0$ and $\BE[(Z_i')^2] \le O(\gamma^2)$, where the expectation and variance are over the randomness of $\Sigma, X, X',$ and $M$.
\end{proposition}

\begin{proof}
    We first prove that $\BE[Z_i'] = 0$: in fact we show that $\BE[Z_i'|\Sigma] = 0$. To do so, note that $M(X_{\sim i})$ and $X_i$ are conditionally independent on $\Sigma$. Therefore, since $\BE[X_i X_i^\top|\Sigma] = \Sigma$, we have 
\[\BE[\langle M(X_{\sim i})-\Sigma, X_i X_i^\top - \Sigma \rangle|\Sigma] = \langle \BE[M(X_{\sim i})|\Sigma]-\Sigma, \BE[X_i X_i^\top|\Sigma] - \Sigma \rangle = 0.\]

    Next, we bound $\BE[(Z_i')^2|\Sigma]$.
    By Proposition~\ref{prop:4th-moment}, replacing $P$ with $M(X_{\sim i})-\Sigma$, we have that
\[\BE\left[\langle M(X_{\sim i})-\Sigma, X_i X_i^\top - \Sigma \rangle^2|\Sigma, X_{\sim i}, M\right] \le 2 \cdot \|\Sigma\|_{op}^2 \cdot \|M(X_{\sim i})-\Sigma\|_F^2,\]
    which means that
\[\BE[(Z_i')^2] \le 2 \cdot \BE\left[\|\Sigma\|_{op}^2 \cdot \|M(X_{\sim i})-\Sigma\|_F^2\right] \le 2 \cdot \sqrt{\BE[\|\Sigma\|_{op}^4] \cdot \BE[\|M(X_{\sim i}) - \Sigma\|_F^4]},\]
    where the last inequality is Cauchy-Schwarz.

    Since $\BE[\|\Sigma\|_{op}^4] \le G_4$ (by \eqref{eq:prior-bound-2}) and $\BE[\|M(X_{\sim i}) - \Sigma\|_F^4] \le \gamma^4$, this means $\BE[(Z_i')^2] \le O(\gamma^2)$.
\end{proof}

\begin{lemma} \label{lem:ub-main}
    Assume the conditions of \Cref{thm:main}. Then, $\BE[Z_i] \le O(\gamma \cdot \eps),$ where the expectation is over the randomness of $\Sigma, X, X', M$.
\end{lemma}

\begin{proof}
    Suppose that we fix $X_1, \dots, X_n$ and $X_i'$. Then, by definition of privacy, for any set $S$, $\BP(M(X) \in S) = e^{\pm \eps} \cdot \BP(M(X_{\sim i} \in S)) \pm \delta$. As a result, since $Z_i$ and $Z_i'$ are the same function applied to $M(X)$ and $M(X_{\sim i})$, respectively (for fixed $X_1, \dots, X_n, X_i'$), this means $\BP(Z_i \in S) = e^{\pm \eps} \cdot \BP(Z_i' \in S) \pm \delta$. Hence, this still holds even when we remove the conditioning on $X_1, \dots, X_n, X_i'$.

    Therefore, by Proposition~\ref{prop:dumb_calculation}, we have that 
\[\left|\BE[Z_i] - \BE[Z_i']\right| \le 2 \eps \cdot \BE[|Z_i'|] + 2 \sqrt{\delta \cdot \BE[Z_i^2 + (Z_i')^2]} \le 2(\eps + \sqrt{\delta}) \cdot \sqrt{\BE[(Z_i')^2]} +  2 \sqrt{\delta \cdot \BE[Z_i^2]}.\]
    So, by Proposition \ref{prop:Zi_prime_bound}, we have that
\begin{equation} \label{eq:b1}
    \BE[Z_i] \le O((\eps+\sqrt{\delta}) \gamma) + 2 \sqrt{\delta \cdot \BE[Z_i^2]}.
\end{equation}
    However, note that $Z_i = \langle M(X) - \Sigma, X_i X_i^\top - \Sigma \rangle,$ which is at most $\|M(X)-\Sigma\|_F \cdot \|X_i X_i^\top - \Sigma\|_F$ in magnitude.
    Hence, 
\begin{equation} \label{eq:b2}
    \BE[Z_i^2] \le \BE[\|M(X)-\Sigma\|_F^2 \cdot \|X_i X_i^\top - \Sigma\|_F^2] \le \sqrt{\BE[\|M(X)-\Sigma\|_F^4] \cdot \BE[\|X_i X_i^\top - \Sigma\|_F^4]}.
\end{equation}

    We know that $\BE[\|M(X)-\Sigma\|_F^4] \le O(\gamma^4)$. Moreover, we can bound
\begin{equation} \label{eq:b4}
    \BE[\|X_i X_i^\top - \Sigma\|_F^4] \le O(\BE[\|X_i X_i^\top\|_F^4 + \BE[\|\Sigma\|_F^4]) = O\left(\BE[\|X_i\|^8] + \BE[\|\Sigma\|_F^4]\right),
\end{equation}
    which by \Cref{prop:conc_bound} with $k = 4$ is at most $O(d^4 \cdot \BE[\|\Sigma\|_{op}^4] + d^2 \cdot \BE[\|\Sigma\|_{op}^4]) \le O(d^4)$. Hence, we can combine Equations~\eqref{eq:b2} and~\eqref{eq:b4} to obtain $\BE[Z_i^2] \le O(\gamma^2 \cdot d^2)$.

    In summary, this means that $\BE[Z_i] \le O(\eps + \sqrt{\delta}) \gamma + O(\sqrt{\delta \cdot \gamma^2 d^2})$.
    By our assumption on $\delta$, this is at most $O(\gamma \cdot \eps)$.
\end{proof}

Because Lemma~\ref{lem:ub-main} holds for all $i \in [n]$, and because of \eqref{eq:prior-bound-2}, we have the following corollary.

\begin{corollary} \label{cor:ub-main}
    Assume the conditions of \Cref{thm:main}. Then, for some constant $C_2 > 0$, we have $\BE\left[\sum_{i=1}^n Z_i\right] \le C_2 \cdot \gamma \cdot \eps \cdot n,$ where the expectation is taken over the randomness of $\Sigma, X, X', M$.
\end{corollary}

\subsection{Lower Bound}

In this section, we prove a lower bound on $\BE\left[\sum_{i=1}^n Z_i\right]$.

Let $\hat{\Sigma} = \frac{1}{n} \sum_{i=1}^n X_i X_i^\top$.
Note that we can write $\BE\left[\sum_{i=1}^n Z_i\right] = n \cdot \langle M(X)-\Sigma, \hat{\Sigma}-\Sigma \rangle,$ so it suffices to prove a lower bound on $\BE[\langle M(X)-\Sigma, \hat{\Sigma}-\Sigma \rangle]$.


\begin{lemma} \label{lem:z-stat-lb}
    For any $\kappa > 0$, we have that
\[
    \BE[\langle M(X)-\Sigma, \hat{\Sigma}-\Sigma\rangle] \ge \BE[\|\Sigma - \hat{\Sigma}\|_F^2] - \sqrt{\BE[\|M(X) - \hat{\Sigma}\|_F^2] \cdot \BE_X\left[\big\|\BE[\Sigma|X]-\hat{\Sigma}\big\|_F^2\right]}.
\]
\end{lemma}

\begin{proof}
    We can rewrite
\begin{equation} \label{eq:lb-basic}
    \langle M(X)-\Sigma, \hat{\Sigma}-\Sigma\rangle = \langle M(X)-\hat{\Sigma}, \hat{\Sigma}-\Sigma \rangle + \langle \hat{\Sigma}-\Sigma, \hat{\Sigma}-\Sigma \rangle = \|\hat{\Sigma}-\Sigma\|_F^2 - \langle M(X)-\hat{\Sigma}, \Sigma-\hat{\Sigma} \rangle.
\end{equation}

    We now bound $\BE[\langle M(X)-\hat{\Sigma}, \Sigma-\hat{\Sigma}\rangle]$. We first consider the distributions of $M(X)-\hat{\Sigma}$ and $\Sigma-\hat{\Sigma}$ conditioned on $X$ (but not conditioned on $\Sigma$). Because the randomness of $M$ is independent of $\Sigma$ conditioned on $X$, we have that $(M(X)-\hat{\Sigma}) \indep (\Sigma-\hat{\Sigma})|X$. So,
\begin{align*}
    \BE[\langle M(X)-\hat{\Sigma}, \Sigma-\hat{\Sigma} \rangle|X]
    &= \big\langle \BE[M(X)|X] - \hat{\Sigma}, \BE[\Sigma|X]-\hat{\Sigma} \big\rangle \\    
    &\le \|\BE[M(X)|X] - \hat{\Sigma}\|_F \cdot \|\BE[\Sigma|X]-\hat{\Sigma}\|_F \\
    &\le \sqrt{\BE_M\|M(X) - \hat{\Sigma}\|_F^2} \cdot \|\BE[\Sigma|X]-\hat{\Sigma}\|_F,
\end{align*}
where the final inequality is by Jensen. By removing the conditioning on $X$ and applying Cauchy-Schwarz, we have
\begin{align*}
    \BE[\langle M(X)-\hat{\Sigma}, \Sigma-\hat{\Sigma} \rangle]
    &\le \BE_X\left[\sqrt{\BE_M\|M(X) - \hat{\Sigma}\|_F^2} \cdot \|\BE[\Sigma|X]-\hat{\Sigma}\|_F\right] \\
    &\le \sqrt{\BE_X \BE_M [\|M(X)-\hat{\Sigma}\|_F^2] \cdot \BE_X [\|\BE[\Sigma|X]-\hat{\Sigma}\|_F^2}] \\
    &= \sqrt{\BE[\|M(X)-\hat{\Sigma}\|_F^2] \cdot \BE_X [\|\BE[\Sigma|X]-\hat{\Sigma}\|_F^2]}.
\end{align*}

    By combining the above bound with Equation~\eqref{eq:lb-basic}, the lemma is complete.
\end{proof}

We now consider computing the posterior distribution $\Sigma|X$. Recall that we chose $p_0 \sim W_d^{-1}(\frac{I}{m}, m)$, which will make the expectation $\BE[\Sigma|X]$ easy to compute (this distribution is the \emph{conjugate prior} for multivariate Gaussians). Under this prior distribution, we have the following lemma.

\begin{lemma} \label{lem:posterior-expectation-bound}
    Suppose $p_0 \sim W_d^{-1}(\frac{I}{m}, m)$, where $m = 2d$. Then, for some absolute constant $C_3 > 0$,
\[\BE_X\left[\|\BE[\Sigma|X]-\hat{\Sigma}\|_F^2\right] \le C_3 \cdot \left(\frac{d^3}{n^2} + \frac{d^4}{n^3}\right).\]
\end{lemma}

\begin{proof}
    Using the PDF of a multivariate Gaussian and Bayes' Rule, the posterior distribution of $\Sigma$ satisfies
\vspace{-10pt}
\begin{align*}
    p(\Sigma|X) &\propto p_0(\Sigma) \cdot \prod_{i=1}^n p(X_i|\Sigma) \\
    &\propto p_0(\Sigma) \cdot \prod_{i=1}^n (\det \Sigma)^{-1/2} \cdot \exp\left(-\frac{1}{2} \langle \Sigma^{-1}, X_i X_i^\top \rangle\right) \\
    &\propto p_0(\Sigma) \cdot \exp\left(- \frac{n}{2}\left(\ln \det \Sigma + \langle \Sigma^{-1}, \hat{\Sigma} \rangle\right)\right).
\end{align*}

By Proposition \ref{prop:wishart-pdf}, the PDF of the prior $p_0$ satisfies 
\[p_0(\Sigma) \propto (\det \Sigma)^{-(m+d+1)/2} \cdot e^{-m/2 \cdot \Tr(\Sigma^{-1})} = \exp\left(-\frac{m+d+1}{2} \cdot \ln \det \Sigma - \frac{m}{2} \cdot \langle \Sigma^{-1}, I\rangle\right).\]
Hence, the the posterior distribution is
\begin{align*}
    p(\Sigma|X) &\propto \exp\left(-\frac{m+d+1}{2} \cdot \ln \det \Sigma - \frac{m}{2} \cdot \langle \Sigma^{-1}, I\rangle - \frac{n}{2} \cdot \ln \det \Sigma - \frac{n}{2} \cdot \langle \Sigma^{-1}, \hat{\Sigma} \rangle\right) \\
    &= (\det \Sigma)^{-(n+m+d+1)/2} \cdot \exp\left(-\frac{1}{2} \cdot \left\langle \Sigma^{-1}, m \cdot I + n \cdot \hat{\Sigma}\right\rangle\right).
\end{align*}
    This means the posterior distribution is again an inverse Wishart distribution: $W_d^{-1}(m \cdot I + n \cdot \hat{\Sigma}, m+n)$. By Proposition~\ref{prop:wishart-expectation}, the expectation of this distribution is is $\frac{m \cdot I + n \cdot \hat{\Sigma}}{m+n-d-1}$.

    Finally, we can write $\BE[\Sigma|X]-\hat{\Sigma} = \frac{m \cdot I + n \cdot \hat{\Sigma}}{m+n-d-1} - \hat{\Sigma} = \frac{m}{m+n-d-1} \cdot I - \frac{m-d-1}{m+n-d-1} \cdot \hat{\Sigma}$. Since $m \ge d+1$, the Frobenius norm of $\BE[\Sigma|X]-\hat{\Sigma}]$ is at most $\|\frac{m}{n} \cdot I\|_F + \|\frac{m}{n} \cdot \hat{\Sigma}\|_F \le \frac{m}{n} \cdot (\|I\|_F + \|\Sigma\|_F + \|\hat{\Sigma}-\Sigma\|_F)$. Therefore,
\begin{align*}
    \BE\left[\big\|\BE[\Sigma|X]-\hat{\Sigma}]\big\|_F^2\right] 
    &\le 3 \left(\frac{m}{n}\right)^2 \cdot \BE\big[\|I\|_F^2 + \|\Sigma\|_F^2 + \|\hat{\Sigma}-\Sigma\|_F^2\big] \\
    &\le 3 \left(\frac{m}{n}\right)^2 \cdot \bigg(d + d \cdot \BE\left[\|\Sigma\|_{op}^2\right] + 3 \cdot \frac{d^2}{n} \cdot \BE\left[\|\Sigma\|_{op}^2\right]\bigg) \\
    &\le O\left(\frac{m^2}{n^2} \cdot \bigg(d + \frac{d^2}{n}\bigg)\right).
\end{align*}
    Above, the first line is by Cauchy-Schwarz, the second follows from \Cref{lem:emp-cov-acc}, and the third follows by \Cref{lem:wishart-operator-norm-bound}.
    Since we are setting $m = 2d$, the proof is complete.
\end{proof}

We are now ready to prove the $\Omega\left(\frac{d^2}{\gamma^2} + \frac{d^2}{\gamma \eps}\right)$ lower bound.
We start by showing $n \ge \Omega\left(\frac{d^2}{\gamma^2}\right).$ While this is purely a non-private lower bound and is essentially folklore, for completeness we prove this lower bound for the specific prior distribution we chose.

\begin{lemma} \label{lem:non-private}
    Under the conditions of Theorem~\ref{thm:main}, where $c_2 \le \sqrt{\frac{g_2}{3 \cdot \max(1, 4C_3/g_2)}}$, then $n \ge \frac{g_2}{3} \cdot \frac{d^2}{\gamma^2}.$
\end{lemma}

\begin{proof}
    We prove that for any deterministic function $f: (\BR^d)^n \to \BR^{d \times d}$ that takes in $X$ and outputs a covariance matrix, we cannot have $\BE_{\Sigma, X}[\|f(X)-\Sigma\|_F^2] \le \gamma^2,$ unless $n \ge \frac{g_2}{3} \cdot \frac{d^2}{\gamma^2}$. Because the randomness of $M$ only depends on $X$ and not directly on $\Sigma$, the claim thus holds for any randomized algorithm $M$.

    Recall that for any random vector $v$ with mean $\mu$, $\BE[\|v\|^2] = \BE[\|v-\mu\|^2] + \|\mu\|^2 \ge \BE[\|v-\mu\|^2]$. Therefore, because the conditional distribution of $f(X)-\Sigma|X$ has mean $f(X)-\BE[\Sigma|X]$, this means that for any fixed $X = \{X_1, \dots, X_n\}$, we have
\[
    \BE[\|f(X)-\Sigma\|_F^2|X] \ge \BE\left[\|(f(X)-\Sigma)-(f(X)-\BE[\Sigma|X])\|_F^2|X\right] = \BE\left[\|\BE[\Sigma|X]-\Sigma\|_F^2|X\right].
\]
    By removing the conditioning on $X$, we thus have that
\[\BE[\|f(X)-\Sigma\|_F^2] \ge \BE\left[\|\BE[\Sigma|X]-\Sigma\|_F^2\right].\]

    Next, we have $\BE[\|\hat{\Sigma}-\Sigma\|_F^2] \le 2 \left(\BE\big[\|\BE[\Sigma|X]-\Sigma\|_F^2\big] + \BE\big[\|\BE[\Sigma|X]-\hat{\Sigma}\|_F^2\big]\right)$, by Cauchy-Schwarz. By Lemmas~\ref{lem:emp-cov-acc} and~\ref{lem:wishart-operator-norm-bound}, we know that $\BE[\|\hat{\Sigma}-\Sigma\|_F^2] \ge 2 g_2 \cdot \frac{d^2}{n},$ and by \Cref{lem:posterior-expectation-bound}, we know that $\BE\big[\|\BE[\Sigma|X]-\hat{\Sigma}\|_F^2\big] \le C_3 \cdot \left(\frac{d^3}{n^2} + \frac{d^4}{n^3}\right)$. Therefore,
\begin{equation} \label{eq:bound}
    \BE\left[\|f(X)-\Sigma\|_F^2\right] \ge g_2 \cdot \frac{d^2}{n} - C_3 \cdot \left(\frac{d^3}{n^2}+\frac{d^4}{n^3}\right).
\end{equation}
    If we set $n = \frac{g_2}{3} \cdot \frac{d^2}{\gamma^2}$, then under the assumption that $\gamma \le \sqrt{\frac{g_2}{3 \cdot \max(1, 4C_3/g_2)}} \cdot \sqrt{d}$, one can verify that $n \ge \max(1, \frac{4 C_3}{g_2}) \cdot d$. Therefore, $C_3 \cdot \left(\frac{d^3}{n^2}+\frac{d^4}{n^3}\right) \le 2C_3 \cdot \frac{d^3}{n^2},$ and $g_2\cdot\frac{d^2}{n}-C_3 \cdot \left(\frac{d^3}{n^2}+\frac{d^4}{n^3}\right) \ge g_2 \cdot \frac{d^2}{n} - 2 C_3 \cdot \frac{d^3}{n^2} \ge \frac{g_2}{2} \cdot \frac{d^2}{n} = \frac{3}{2} \cdot \gamma^2$. Thus, \eqref{eq:bound} implies that $\BE[\|f(X)-\Sigma\|_F^2] \ge \frac{3}{2} \gamma^2$ for any function $f$ of $\frac{g_2}{3} \cdot \frac{d^2}{\gamma^2}$ samples, which concludes the proof.
\end{proof}

We are now ready to prove the main result.

\begin{proof}[Proof of \Cref{thm:main}]
    We set $c_2 = \min\left(\sqrt{\frac{g_2}{3 \cdot \max(1, 4C_3/g_2)}}, \frac{g_2}{2 \sqrt{C_3(1+(9 G_2/g_2))}}\right)$ and $c_3 = \min\left(\frac{g_2}{3}, \frac{g_2}{C_2}\right)$.
    Now, assume that $\BE[\|M(X)-\Sigma\|_F^2] \le \gamma^2$.
    Note that by Lemmas \ref{lem:emp-cov-acc} and \ref{lem:wishart-operator-norm-bound}, we have that $3G_2 \cdot \frac{d^2}{n} \ge \BE[\|\Sigma-\hat{\Sigma}\|_F^2] \ge 2g_2 \cdot \frac{d^2}{n}$.
    By first using \Cref{lem:z-stat-lb}, and then Cauchy-Schwarz, and then our bounds on $\BE[\|\Sigma-\hat{\Sigma}\|_F^2]$ and $\BE[\|M(X)-\Sigma\|_F^2]$ along with \Cref{lem:posterior-expectation-bound}, we have that
\begin{align*}
    \BE[\langle M(X)-\Sigma, \hat{\Sigma}-\Sigma \rangle]
    &\ge \BE[\|\Sigma - \hat{\Sigma}\|_F^2] - \sqrt{\BE[\|M(X) - \hat{\Sigma}\|_F^2] \cdot \BE_X\left[\big\|\BE[\Sigma|X]-\hat{\Sigma}\big\|_F^2\right]} \\
    &\ge \BE[\|\Sigma - \hat{\Sigma}\|_F^2] - \sqrt{2 \left(\BE[\|M(X) - \Sigma\|_F^2+\|\Sigma-\hat{\Sigma}\|_F^2]\right) \cdot \BE_X\left[\big\|\BE[\Sigma|X]-\hat{\Sigma}\big\|_F^2\right]} \\
    &\ge 2g_2 \cdot \frac{d^2}{n} - \sqrt{2 \left(\gamma^2 + 3 G_2 \cdot \frac{d^2}{n}\right) \cdot C_3 \left(\frac{d^3}{n^2}+\frac{d^4}{n^3}\right)}.
\end{align*}
    Now, by \Cref{lem:non-private}, we can assume $n \ge \frac{g_2}{3} \cdot \frac{d^2}{\gamma^2}$. Moreover, since $c_2 \le \sqrt{\frac{g_2}{3 \cdot \max(1, 4C_3/g_2)}}$, then $n \ge d,$ so $\frac{d^4}{n^3} \le \frac{d^3}{n^2}$. So, we can further bound the above as
\[\BE[\langle M(X)-\Sigma, \hat{\Sigma}-\Sigma \rangle] \ge 2g_2 \cdot \frac{d^2}{n} - \sqrt{4 \gamma^2 \cdot \left(1 + \frac{9 G_2}{g_2}\right) \cdot C_3 \cdot \frac{d^3}{n^2}}.\]
    Finally, because $\gamma \le c_2 \sqrt{d}$ and $c_2 \le \frac{g_2}{2 \sqrt{C_3(1+(9 G_2/g_2))}}$, this means that 
\[\sqrt{4 \gamma^2 \cdot \left(1 + \frac{9 G_2}{g_2}\right) \cdot C_3 \cdot \frac{d^3}{n^2}} \le 2 c_2 \sqrt{d} \cdot \sqrt{\left(1 + \frac{9 G_2}{g_2}\right) \cdot C_3 \cdot \frac{d^3}{n^2}} \le g_2 \cdot \frac{d^2}{n}.\]
    So, $\BE[\langle M(X)-\Sigma, \hat{\Sigma}-\Sigma \rangle] \ge g_2 \cdot \frac{d^2}{n}$.
    
    Conversely, from \Cref{cor:ub-main}, we know that $\BE[\langle M(X)-\Sigma, \hat{\Sigma}-\Sigma \rangle] = \frac{1}{n} \cdot \BE[\sum_{i=1}^{n} Z_i] \le C_2 \cdot \gamma \cdot \eps$. Thus, we must have $g_2 \cdot \frac{d^2}{n} \le C_2 \gamma \eps$, so this means $n \ge \frac{g_2}{C_2} \cdot \frac{d^2}{\gamma \eps}$. Finally, since $n \ge \frac{g_2}{3} \cdot \frac{d^2}{\gamma^2}$ by \Cref{lem:non-private}, the proof is complete.
\end{proof}

\section{Lower Bound for Private Heavy-Tailed Mean Estimation} \label{sec:heavy-tailed}

In this section, we prove the following theorem, which formalizes \Cref{thm:heavy-tailed}.

\begin{theorem} \label{thm:heavy-tailed-formal}
    Let $C_1 > 1$ be a sufficiently large absolute constant. Fix $k \ge 2$, and let $C(k)$ be a sufficiently large constant that only depends on $k$. Suppose $\alpha, \eps \le 1$, and $\delta \le \tilde{\Omega}\left(\frac{\alpha \eps}{d}\right)^{C_1}$.
    Suppose that $M$ is an $(\eps, \delta)$-DP mechanism that takes as input $X = \{X_1, \dots, X_n\}$ where each $X_i \in \BR^d$.
    Suppose that if $X_1, \dots, X_n \sim \mathcal{D}$, for any distribution $\mathcal{D}$ where $\BE[|\langle \mathcal{D}-\BE[\mathcal{D}], v \rangle|^k] \le C(k)$ for all unit vectors $v$, then with probability at least $2/3$ over the randomness of $M$ and $X_1, \dots, X_n$, $\|M(X)-\BE[\mathcal{D}]\|_2 \le \alpha$.
    Then, $n \ge \tilde{\Omega}\left(\frac{d}{\alpha^2} + \frac{d}{\alpha^{k/(k-1)} \cdot \eps}\right)$.
\end{theorem}

Let $T$ be a parameter that we wil choose later.
Fix parameters $\eps, \delta \le 1$ and $\alpha \le \frac{1}{T}$ (we deal with the case when $\frac{1}{T} < \alpha \le 1$ later). Let $\beta = \alpha \cdot T$, and for each $\mu$ with $\|\mu\|_2 \le 1$, define the mixture distribution $\cD_{\mu}$ to be drawn as $\cN(\beta^{-\frac{1}{k-1}} \cdot \mu, \beta^{-\frac{2}{k-1}} \cdot I)$ with probability $\beta^{\frac{k}{k-1}}$, and $\textbf{0}$ otherwise.

It is simple to see that $\BE[\cD_\mu]=\beta \cdot \mu$. We now show that $\cD_\mu$ has bounded $k$th absolute moment, i.e., $\BE[|\langle \cD-\BE[\cD], v \rangle|^k] \le C(k)$ for all unit vectors $v$, for an appropriately chosen $C(k)$.

\begin{proposition}
    Suppose that $0 < \beta \le 1,$ $\|\mu\|_2\le 1$, and $\cD_\mu$ has the mixture distribution as above. Then, $\cD_\mu$ has bounded $k$th absolute moment.
\end{proposition}

\begin{proof}
    In any unit direction $v$, the distribution $\cD_\mu$ projected onto $v$ equals $\cN(\beta^{-\frac{1}{k-1}} \cdot \langle \mu, v \rangle, \beta^{-\frac{2}{k-1}} \cdot I)$ with $\beta^{\frac{k}{k-1}}$ probability, and $0$ otherwise. So, $\langle \cD_\mu - \BE[\cD_\mu], v \rangle$ equals $\cN((\beta^{-\frac{1}{k-1}}- \beta) \cdot \langle \mu, v \rangle, \beta^{-\frac{2}{k-1}} \cdot I)$ with $\beta^{\frac{k}{k-1}}$ probability, and $-\beta \cdot \langle \mu, v \rangle$ otherwise.

    It is well-known that for any fixed $k$, the $k$th absolute moment of $\cN(x, \sigma^2)$ is at most $C^{(0)}(k) \cdot (|x|+\sigma)^k$ for some constant $C^{(0)}(k)$ only depending on $k$. Hence, because $\langle \mu, v \rangle \le 1$, and because $\beta \le 1$, we can bound the $k$th absolute moment of $\langle \cD_\mu - \BE[\cD_\mu], v \rangle$ as at most
\[\beta^{\frac{k}{k-1}} \cdot \left(C^{(0)}(k) \left(|\beta^{-\frac{1}{k-1}}-\beta| + \beta^{-\frac{1}{k-1}}\right)^k\right) + \beta^k \le \beta^{\frac{k}{k-1}} \cdot \left(C^{(0)}(k) \left(2 \beta^{-\frac{1}{k-1}}\right)^k\right) + \beta^k \le 2^k C^{(0)}(k)+1.\]
    So, we can set $C(k) = 2^k C^{(0)}(k)+1.$
\end{proof}

Now, to estimate the mean $\BE[\cD_\mu] = \beta \cdot \mu$ up to error $\alpha$, it is equivalent to estimate $\mu$ up to error $\frac{\alpha}{\beta}= \frac{1}{T}.$ We show the following reduction lemma.

\begin{lemma} \label{lem:reduction-1}
    Suppose that $M$ is an $(\eps, \delta)$-DP mechanism that, given $n = m/(100 \cdot \beta^{k/(k-1)})$ samples from $\cD_\mu$, for unknown $\|\mu\|_2 \le 1$, can learn $\mu$ up to error $\frac{1}{T}$, with failure probability $\phi$. Then, there exists an $(\eps, \delta)$-DP mechanism $\overline{M}$ that, given $m$ samples from $\cN(\mu, I)$, for unknown $\|\mu\|_2 \le 1$, can learn $\mu$ up to error $\frac{1}{T}$, with failure probability $\phi+0.01$.
\end{lemma}

\begin{proof}
    Consider drawing $m$ samples $X_1, \dots, X_m$. We convert this into samples $Y_1, \dots, Y_n$ as follows. We draw $n$ Bernoulli variables $z_1, \dots, z_n \sim Bern(\beta^{k/(k-1)})$. Now, if $z_1 = 1$, we set $Y_1 = X_1$, and otherwise, we set $Y_1 = \textbf{0}.$ For each $z_i,$ if either we have already used up $X_1, \dots, X_m$ or if $z_i = 0$, we set $Y_i = \textbf{0}$. Otherwise, we set $Y_i = X_j,$ where we have so far used up $X_1, \dots, X_{j-1}$. Finally, our algorithm $\overline{M},$ on data $X_1, \dots, X_m$, outputs $M(Y_1, \dots, Y_n)$.

    To see why this is private, let $X, X'$ be adjacent datasets. If we couple the randomness of the Bernoulli variables, then at most one data point can change in our creation of $Y_1, \dots, Y_n$. Therefore, $\overline{M}$ is $(\eps, \delta)$-DP, since $M$ is $(\eps, \delta)$-DP.

    Next, we consider the accuracy of the mechanism. Note that $\BE[Bin(n, \beta^{k/(k-1)})] = m/100$, and so by Markov's inequality, $\BP(Bin(n, \beta^{k/(k-1)}) \le m) \ge 0.99$. In this case, we do not run out of samples, so in fact the distribution of $Y_1, \dots, Y_n$ is precisely the mixture distribution $\cN(\beta^{-\frac{1}{k-1}} \cdot \mu, \beta^{-\frac{2}{k-1}} \cdot I)$ with probability $\beta^{-\frac{k}{k-1}}$ and $0$ otherwise. Hence, this algorithm learns $\mu$ up to error $\frac{1}{T}.$
\end{proof}

Now, we explain how the known bounds on Gaussian estimation imply that any such algorithm $\overline{M}$ will need sufficiently many samples.

\begin{lemma} \label{lem:reduction-2}
    There exists a sufficiently large constant $C_6$ such that, for $T = \left(\log \frac{d}{\eps}\right)^{C_6}$ and $\delta = \left(\frac{\eps}{d \cdot T}\right)^{C_1},$ any $(\eps, \delta)$-DP mechanism $\overline{M}$ that, given $m$ samples from $\cN(\mu, I)$ where $\|\mu\|_2 \le 1$, that can learn $\mu$ up to error $\frac{1}{T}$ requires at least $m \ge \frac{d}{\eps}$ samples.
\end{lemma}

\begin{proof}
    Suppose there existed such an algorithm. Then, we can learn the mean $\mu$ up to error $\frac{1}{T}$, even if we were not promised $\|\mu\|_2 \le 1$. This is because for some $C_5$, we can first, using $\left(\frac{d}{\eps} (\log \frac{d}{\eps \delta})^{C_5}\right)$ samples, learn $\mu$ up to $\ell_2$ error $1$ with $(\eps, \delta)$-DP, by \Cref{thm:mean}. Next, we can learn $\mu$ up to error $\frac{1}{T}$ using $m$ fresh samples. Since each stage is $(\eps, \delta)$-DP, and we use fresh samples, the overall algorithm is $(\eps, \delta)$-DP. However, learning $\mu$ up to error $\frac{1}{T}$ would require at least $\left(\frac{d \cdot T}{\eps} (\log \frac{d \cdot T}{\eps})^{-C_5}\right)$ samples by \Cref{thm:mean}, which means that $m + \left(\frac{d}{\eps} (\log \frac{d}{\eps \delta})^{C_5}\right) \ge \left(\frac{d \cdot T}{\eps} (\log \frac{d}{\eps})^{-C_5}\right)$. So, if $T \ge 2 \left(\log \frac{d}{\eps \delta}\right)^{2C_5}$, then $m \ge \frac{d}{\eps} (\log \frac{d}{\eps \delta})^{C_5} \ge \frac{d}{\eps}$. But this condition on $T$ is equivalent to $T \ge 2 \left(\log \frac{d}{\eps} + C_1 \log \frac{d \cdot T}{\eps}\right)^{C_5}$, which holds for some $T = \left(\log \frac{d}{\eps}\right)^{C_6}$, where $C_6$ is sufficiently large and depends on $C_1$ and $C_5$.
\end{proof}

By combining Lemmas~\ref{lem:reduction-1} and \ref{lem:reduction-2}, this implies that learning $\mu$ up to error $\frac{1}{T},$ or equivalently, learning the mean of $\cD_\mu$ up to error $\alpha \le \frac{1}{T}$, with $(\eps, (\frac{\eps}{d T})^{C_1})$-DP, requires at least $\frac{d}{100 \eps \cdot \beta^{k/(k-1)}} = \frac{1}{100 (\log (d/\eps))^{C_6 \cdot k/(k-1)}} \cdot \frac{d}{\eps \cdot \alpha^{k/(k-1)}} = \tilde{\Omega}\left(\frac{d}{\alpha^{k/(k-1)} \cdot \eps}\right)$. Alternatively, if $\frac{1}{T} \le \alpha \le 1$, then \Cref{thm:mean} implies that even learning an identity-covariance Gaussian up to error $1$ with $(\eps, (\frac{\eps}{d})^{C_1})$-DP requires $\tilde{\Omega}\left(\frac{d}{\eps}\right)$ samples. Since $1 \le \alpha^{-1} \le \log(\frac{d}{\eps})^{C_6}$ in this case, we still have the number of samples is $\tilde{\Omega}\left(\frac{d}{\alpha^{k/(k-1)} \eps}\right)$. In either case, learning a distribution with bounded $k$th moments, with $\left(\eps, \left(\frac{\eps}{d}\right)^{C_1}/(\log \frac{d}{\eps})^{C_1 C_6}\right)$-DP requires $\tilde{\Omega}\left(\frac{d}{\alpha^{k/(k-1)} \eps}\right)$ samples. Moreover, since even learning an identity-covariance Gaussian without privacy requires $\Omega\left(\frac{d}{\alpha^2}\right)$ samples, this completes the proof of Theorem~\ref{thm:heavy-tailed-formal}.

\section{Lower Bound for Empirical Covariance Estimation} \label{sec:empirical}

In this section, we show how our result on Gaussian covariance estimation also implies an improved result for empirical covariance estimation.

We recall the setup: we are given $n$ points $X_1, \dots, X_n$ in the unit ball in $d$-dimensions, and our goal is to estimate the empirical covariance $\hat{\Sigma}$, which equals $\frac{1}{n} \sum_{i=1}^n (X_i-\hat{\mu}) (X_i-\hat{\mu})^\top$, for $\hat{\mu} = \frac{1}{n} \sum_{i=1}^n X_i$. Equivalently, $\hat{\Sigma} = \frac{1}{n} \sum_{i=1}^n X_i X_i^\top - \hat{\mu} \hat{\mu}^\top$. We wish to privately find some $\tilde{\Sigma}$ such that $\|\tilde{\Sigma}-\hat{\Sigma}\|_F$ is small in expectation.
Our goal will be to show that if $\tilde{O}(n^{3/2}) \le d \le \tilde{\Omega}(n^2)$, then any $(\eps, \delta)$-DP algorithm, even for $\eps$ a fixed constant, cannot have expected Frobenius error less than $\tilde{\Omega}(\frac{d}{n}).$

\begin{lemma} \label{lem:reduction-3}
    Suppose that $M$ is an $(\eps, \delta)$-DP mechanism that, for any $n$ samples $X_1, \dots, X_n$ in the ball of radius $1$, can learn $\hat{\Sigma} = \frac{1}{n} \sum_{i=1}^n X_i X_i^\top - \hat{\mu} \hat{\mu}^\top$ up to Frobenius error $\alpha$, with failure probability $\phi$. Then, there exists an $(\eps, \delta)$-DP mechanism $\overline{M}$ that, if given $n$ samples $X_1, \dots, X_n \sim \cN(0, \Sigma)$ for any $\|\Sigma\|_{op} \le 10$, can learn $\Sigma$ up to Frobenius error $O\left(d \cdot \alpha + \frac{d}{\sqrt{n}}\right)$, with failure probability $\phi+n \cdot e^{-\Omega(d)} + 0.02$.
\end{lemma}

\begin{proof}
    The algorithm $\overline{M}$ works as follows. Given data $X_1, \dots, X_n$, let $Y_i = \frac{X_i}{20 \sqrt{d}}$ if $\|X_i\|_2 \le 20 \sqrt{d}$, and $Y_i = 0$ otherwise. Now, $\overline{M}$ simply outputs $400d \cdot M(Y_1, \dots, Y_n)$. Since each $Y_i$ is a deterministic function of $X_i$, and $\|Y_i\|_2 \le 1$ for all $i$, this implies that $\overline{M}$ is also $(\eps, \delta)$-DP, so we just need to verify accuracy.

    Suppose $X_1, \dots, X_n \sim \cN(0, \Sigma)$, where $\|\Sigma\|_{op} \le 10$. Then, by the Hanson-Wright inequality (\Cref{lem:hanson-wright}), we have that $\BP(\|X_i\| \le 20 \sqrt{d}) \ge 1-e^{-\Omega(d)}$ for each $i$, so the probability that some $Y_i$ is not $\frac{X_i}{20 \sqrt{d}}$ is at most $n \cdot e^{-\Omega(d)}$.

    Assuming this event does not hold, then with failure probability at most $\phi$, 
\[\left\|M(Y_1, \dots, Y_n)-\left(\frac{1}{n}\sum_{i=1}^n Y_i Y_i^\top - \bar{Y} \bar{Y}^\top\right)\right\|_F \le \alpha,\]
    where $\bar{Y} = \frac{1}{n} \sum_{i=1}^n Y_i$. By scaling by $400 d$, this implies that
\[\left\|\overline{M}(X_1, \dots, X_n)-\left(\frac{1}{n}\sum_{i=1}^n X_i X_i^\top - \bar{X} \bar{X}^\top\right)\right\|_F \le 400 d \cdot \alpha.\]
    Next, $\|\frac{1}{n} \sum X_i X_i^\top - \Sigma\|_F \le O(\sqrt{\frac{d^2}{n}})$ with at least $0.99$ probability by \Cref{lem:emp-cov-acc}, and $\|\bar{X} \bar{X}^\top\|_F = \|\bar{X}\|_2^2,$ where $\bar{X} \sim \cN(0, \frac{\Sigma}{n})$. With at least $0.99$ probability, $\|\bar{X}\|_2^2 \le O(\Tr(\frac{\Sigma}{n})) \le O(\frac{d}{n})$. Therefore, 
\[\|\overline{M}(X_1, \dots, X_n) - \Sigma\|_F \le 400 d \cdot \alpha + O\left(\sqrt{\frac{d^2}{n}} + \frac{d}{n}\right) = O\left(d \cdot \alpha + \frac{d}{\sqrt{n}}\right).\]
    This concludes the proof.
\end{proof}

Now, by \Cref{thm:covariance-formal} (modified with assumption~\ref{eq:assumption}), for any $\gamma \le c_2 \sqrt{d}, \eps \le 1, \delta \le \frac{\eps^2}{d^2},$ and $\|\Sigma\|_{op} \le 10$, any $(\eps, \delta)$-DP algorithm that satisfies $\BP_{X, M}(\|M(X)-\Sigma\|_F \le \gamma) \ge 2/3$ for $X = \{X_1, \dots, X_n\} \sim \cN(0, \Sigma)$ must use $n \ge \tilde{\Omega}\left(\frac{d^2}{\gamma \cdot \eps}\right)$ samples. We will use \Cref{lem:reduction-3} to convert the lower bound of \Cref{thm:covariance-formal} into a lower bound for empirical covariance estimation.

\begin{corollary} \label{cor:emp-cov}
    Suppose that $n/(\log n)^{O(1)} \ge d \ge \eps \sqrt{n} \cdot (\log n)^{O(1)}$, and $\delta \le \frac{\eps^2}{d^2}$. Then, if an $(\eps, \delta)$-DP algorithm can learn the empirical covariance of $X_1, \dots, X_n$ up to Frobenius error $\alpha$, with probability at least $3/4$, then $\alpha \ge \tilde{\Omega}\left(\min\left(\frac{1}{\sqrt{d}}, \frac{d}{\eps n}\right)\right).$
\end{corollary}

\begin{proof}
    Suppose otherwise, which means $\alpha \le \tilde{\Omega}\left(\frac{1}{\sqrt{d}}\right).$ By \Cref{lem:reduction-3}, we can privately learn $\Sigma$ up to Frobenius error $\gamma = O(d \alpha +\frac{d}{\sqrt{n}})$, given $n$ samples from $X_1, \dots, X_n$, as long as $\|\Sigma\|_{op} \le 10$. Since we are assuming that $\alpha \le \tilde{\Omega}\left(\frac{1}{\sqrt{d}}\right)$ and $d \le n/(\log n)^{O(1)}$, this means that $\gamma \le \sqrt{d}/(\log n)^{O(1)} \le c_2 \sqrt{d}$. Therefore, we must have that $n \ge \tilde{\Omega}\left(\frac{d^2}{\gamma \eps}\right) \ge \tilde{\Omega}\left(\frac{d^2}{(d \alpha + d/\sqrt{n}) \cdot \eps}\right) \ge \tilde{\Omega}\left(\min\left(\frac{d}{\alpha \eps}, \frac{d \sqrt{n}}{\eps}\right)\right).$ Since $d \ge \eps \sqrt{n} \cdot \poly\log n$, this implies that $n < \tilde{\Omega}\left(\frac{d \sqrt{n}}{\eps}\right)$, so we must have $n \ge \tilde{\Omega}\left(\frac{d}{\alpha \eps}\right)$, which means $\alpha \ge \tilde{\Omega}\left(\frac{d}{\eps n}\right)$.
\end{proof}

Finally, we remark that the problem of empirical covariance estimation up to Frobenius norm error $\alpha$ cannot get easier as the dimension increases. Indeed, if one can privately estimate covariance in $d' \ge d$ dimensions, then one can achieve the same accuracy and privacy in $d$ dimensions, simply by embedding $d$ dimensions into the first $d$ coordinates of the $d'$-dimensional space. Hence, for $n/(\log n)^{O(1)} \ge d \ge (\eps n)^{2/3},$ we can use the bound that $\alpha \ge \tilde{\Omega}\left(\frac{1}{(\eps n)^{1/3}}\right)$. In other words, for all $n/(\log n)^{O(1)} \ge d \ge \eps \sqrt{n} \cdot (\log n)^{O(1)}$, we in fact have that $\alpha \ge \tilde{\Omega}\left(\min\left(\frac{1}{(\eps n)^{1/3}}, \frac{d}{\eps n}\right)\right)$.

When considering the case of $\eps = 1$, we can prove the following improved lower bound on empirical covariance estimation, when $\tilde{O}(\sqrt{n}) \le d \le n^{4/3}$.

\begin{corollary}
    Suppose that $d \ge \tilde{O}(\sqrt{n})$. Then, any $(1, \frac{1}{d^2})$-DP algorithm that can learn the empirical covariance of any $X_1, \dots, X_n$ in the unit ball up to Frobenius error $\alpha$, with probability at least $3/4$, must satisfy $\alpha \ge \tilde{\Omega}(\frac{d}{n})$ when $d \le n^{2/3}$, and $\alpha \ge \tilde{\Omega}(\frac{1}{n^{1/3}})$ when $d \ge n^{2/3}$.
\end{corollary}

\begin{proof}
    First, suppose $\tilde{O}(\sqrt{n}) \le d \le n^{2/3}$. In this case, we have that $d \le n/(\log n)^{O(1)}$, so we can apply~\Cref{cor:emp-cov} to say that $\alpha \ge \tilde{\Omega}\left(\min\left(\frac{1}{\sqrt{d}}, \frac{d}{n}\right)\right) = \tilde{\Omega}\left(\frac{d}{n}\right)$. Moreover, given any $(1, \frac{1}{d^2})$-DP algorithm that works in $d \ge n^{2/3}$ dimensions, it is also $(1, \frac{1}{(n^{2/3})^2})$-DP, and must work in $d = n^{2/3}$ dimensions. Therefore, the error must be at least $\alpha \ge \tilde{\Omega}(\frac{n^{2/3}}{n}) \ge \tilde{\Omega}\left(\frac{1}{n^{1/3}}\right)$.
\end{proof}

In combination with the known lower bounds for empirical covariance estimation~\cite{KasiviswanathanRSU10, KamathLSU19, DongLY22}, this implies that the error is at least $\tilde{\Omega}\left(\frac{d}{n}\right)$ for $d \le n^{2/3}$, $\tilde{\Omega}\left(\frac{1}{n^{1/3}}\right)$ for $n^{2/3} \le d \le n^{4/3}$, and $\tilde{\Omega}\left(\frac{\sqrt{d}}{n}\right)$ for $n^{4/3} \le d \le n^2$.

\section*{Acknowledgments}

I would like to thank Gautam Kamath and Argyris Mouzakis for helpful conversations and pointing me to some relevant references.

\newcommand{\etalchar}[1]{$^{#1}$}


\appendix

\section{Omitted Proofs} \label{sec:prelim-app}
\subsection{Omitted proofs from \Cref{sec:prelim}}

First, we prove \Cref{lem:wishart-operator-norm-bound}.

\begin{proof}[Proof of \Cref{lem:wishart-operator-norm-bound}]
    Note that $M^{-1} \sim W_d(I, m)$. By the proof of~\cite[Lemma 4.1]{ChenD05}, we have that
\[\BP\left(\lambda_{\min}(M^{-1}) \le \frac{m}{x} \right) \le \frac{1}{\Gamma(m-d+2)} \cdot \left(\frac{m}{\sqrt{x}}\right)^{m-d+1} = \frac{(m/\sqrt{x})^{m-d+1}}{(m-d+1)!}.\]
    It is well known that $n! \ge (n/e)^n$ for all positive integers $n$, so this is at least $(e/\sqrt{x})^{m-d+1}$. Therefore, since $\lambda_{\max}(M) = \lambda_{\min}(M^{-1})^{-1}$, we have $\BP(\lambda_{\max}(M) \ge \frac{x}{m}) \le (e/\sqrt{x})^{m-d+1} \le (e/\sqrt{x})^d = (e^2/x)^{d/2}$. Therefore, as long as $d/2 \ge 5k,$ we have that $m \cdot \lambda_{\max}(M)$ has its $k$th moment bounded by some constant $C_k$.

    Next, by \cite[Theorem 4.4.5]{VershyninBook}, if $A$ is a $m \times d$-dimensional matrix with i.i.d. $\cN(0, 1)$ entries, then $\|A\|_{op} \le O(\sqrt{m}+\sqrt{d})$ with high probability. Therefore, $(A^\top A)^{-1}$ has smallest eigenvalue at least $\Omega((\sqrt{m}+\sqrt{d})^{-2})$ with high probability. However, note that $(A^\top A)^{-1} \sim W_d^{-1}(I, m)$, and since $m \ge d$, $(\sqrt{m}+\sqrt{d})^{-2} \ge \frac{1}{4m}$. Thus, $\lambda_{\min}(M) \ge \Omega\left(\frac{1}{m}\right)$ with at least $2/3$ probability.
\end{proof}

Next, we note (and prove) the following auxiliary proposition.

\begin{proposition} \label{prop:frob_op_product}
    Let $P, J$ be symmetric matrices of the same size. Then, $\|J P J\|_F \le \|P\|_F \cdot \|J\|_{op}^2$.
\end{proposition}

\begin{proof}
    For any (not necessarily symmetric) square matrix $X$ of the same size as $J$, note that $\|J X\|_F^2 = \sum \|J X_i\|_2^2,$ where $X_i$ is the $i$th column of $X$. This is at most $\sum_i \|J\|_{op}^2 \cdot \|X_i\|_2^2 = \|J\|_{op}^2 \cdot \|X\|_F^2$. Thus, $\|J X\|_F^2 \le \|J\|_{op}^2 \cdot \|X\|_F^2$.

    Therefore, $\|J P J\|_F^2 \le \|J\|_{op}^2 \cdot \|P J\|_F^2 = \|J\|_{op}^2 \cdot \|J P\|_F^2 \le \|J\|_{op}^4 \cdot \|P\|_F^2,$ where the middle equality holds because $PJ = (JP)^\top$ when $J, P$ are symmetric. Since operator norms and Frobenius norms are always positive, we thus have $\|J P J\|_F \le \|J\|_{op}^2 \cdot \|P\|_F.$
\end{proof}

We can now prove Proposition~\ref{prop:4th-moment}.

\begin{proof}[Proof of \Cref{prop:4th-moment}]
    First, assume that $P$ is symmetric. Write $X = \Sigma^{1/2} Y$, where $Y \sim \cN(0, I)$. We can write
\begin{align*}
    \langle P, X X^\top - \Sigma \rangle &= \Tr(P \cdot X X^\top - P \Sigma)\\
    &= \Tr(P \cdot \Sigma^{1/2} Y Y^\top \Sigma^{1/2} - P \Sigma)\\
    &= \Tr((\Sigma^{1/2} P \Sigma^{1/2}) \cdot (Y Y^\top-I))\\
    &= \langle \Sigma^{1/2} P \Sigma^{1/2}, Y Y^\top - I \rangle.
\end{align*}
    If we write $Q := \Sigma^{1/2} P \Sigma^{1/2}$, then $Q$ is symmetric, so $\langle Q, Y Y^\top - I \rangle = \sum_i Q_{ii} (Y_i^2-1) + \sum_{i < j} Q_{ij} (2Y_i Y_j)$. Note that each $Y_i^2-1$ and $2Y_i Y_j$ has mean $0$, variance $2$, and are pairwise uncorrelated (though not necessarily independent). Therefore, $\langle Q, Y Y^\top - I \rangle$ has expectation $0$ and variance at most $2\cdot \sum_{i, j} Q_{i, j}^2 = 2 \cdot \|\Sigma^{1/2} P \Sigma^{1/2}\|_F^2$. Then, by Proposition \ref{prop:frob_op_product}, this is at most $2 \cdot \|\Sigma\|_{op}^2 \cdot \|P\|_F^2$.

    Finally, if $P$ is not symmetric, let $P' = \frac{P+P^\top}{2}$. Note that $\langle P, XX^\top - \Sigma \rangle = \langle P', XX^\top - \Sigma \rangle$ but $\|P'\|_F^2 \le \|P\|_F^2$. So, the claim still holds.
\end{proof}

Next, we prove \Cref{prop:conc_bound}.

\begin{proof}[Proof of \Cref{prop:conc_bound}]
    Note that $X = \Sigma^{1/2} \cdot Z$, where $Z \sim \cN(0, I)$. Note that $\|X\|_2 \le \|\Sigma\|_{op}^{1/2} \cdot \|Z\|_2$. Also, for any integer $C \ge 2$, $\BP(\|Z\|^2 \ge C \cdot d) \le e^{-c \cdot (C-1) d} \le e^{-c/2 \cdot C \cdot d} \le e^{-c/2 \cdot C \cdot d}.$ So, $\BP(\|X\|^2 \ge C \cdot \|\Sigma\|_{op} \cdot d) \le e^{-c/2 \cdot C}$, so by Lemma~\ref{lem:hanson-wright},
\[\BP\left(\frac{\|X\|^2}{C \cdot \|\Sigma\|_{op} \cdot d}\right) \le e^{-c/2 \cdot C}.\]
    Hence, $\frac{\|X\|^2}{C \cdot \|\Sigma\|_{op} \cdot d}$ has an exponential tail bound with an absolute constant, so $\BE\left[\left(\frac{\|X\|^2}{C \cdot \|\Sigma\|_{op} \cdot d}\right)^k\right] \le c_k$ for some constant $c_k$, for any integer $k \ge 1$. By rearranging,the proof is complete.
\end{proof}

Finally, we prove \Cref{prop:dumb_calculation}.

\begin{proof}[Proof of \Cref{prop:dumb_calculation}]
    Assume that $X, Y$ have finite probability density functions $p_X$ and $p_Y$. (This assumption is WLOG since we may add an arbitrarily small Gaussian to $X$ and $Y$ for smoothing.) Then, since $1-\eps \le e^{-\eps}$ and $e^{\eps} \le 1+2\eps$ for $\eps \le 1$, we can write $p_Y(x) = p_X(x) \cdot (1 + a(x)) + b(x)$, where $|a(x)| \le 2 \eps$ for all $x$ and $\int_{-\infty}^{\infty} |b(x)| dx \le \delta$. Now, we can write
\[\BE[X-Y] = \int_{-\infty}^{\infty} x \cdot (p_X(x) \cdot a(x) + b(x)) dx,\]
    which in absolute value is bounded by
\begin{equation} \label{eq:dumb_calculation_1}
    \int_{-\infty}^{\infty} p_X(x) \cdot |x| \cdot 2 \eps dx + \int_{-\infty}^{\infty} |x| \cdot |b(x)| dx = 2 \eps \cdot \BE[|X|] + \int_{-\infty}^{\infty} |x| \cdot |b(x)| dx.
\end{equation}
    Since $b(x) = p_Y(x)-(p_X(x) \cdot (1+a(x))),$ we can use Cauchy-Schwarz to bound 
\begin{align}
    \int_{-\infty}^{\infty} |x| \cdot |b(x)| dx &\le \sqrt{\int_{-\infty}^{\infty} |b(x)| dx \cdot \int_{-\infty}^{\infty} x^2 \cdot |b(x)| dx} \nonumber \\
    &\le \sqrt{\delta \cdot \int_{-\infty}^{\infty} x^2 \cdot (p_Y(x) - p_X(x) \cdot (1+a(x))) dx} \nonumber \\
    &\le \sqrt{\delta \cdot \int_{-\infty}^{\infty} x^2 \cdot p_Y(x) dx + \int_{-\infty}^{\infty} x^2 \cdot p_X(x) \cdot (1+2\eps) dx} \nonumber \\
    &= \sqrt{\delta \cdot (\BE[Y^2] + (1+2 \eps) \cdot \BE[X^2])}. \label{eq:dumb_calculation_2}
\end{align}
    By combining Equations \eqref{eq:dumb_calculation_1} and \eqref{eq:dumb_calculation_2}, and since $\eps \le 1$, this completes the proof.
\end{proof}

\subsection{Reduction of Assumptions} \label{subsec:assumptions}

We describe how to convert an algorithm $M$ satisfying \eqref{eq:assumption} into an algorithm $M$ satisfying \eqref{eq:assumption-stronger}, at the cost of a slight increase in the number of samples.

\begin{lemma}
    Suppose that $M$ is an $(\eps, \delta)$-DP algorithm, that uses $n$ samples and satisfies \eqref{eq:assumption}.
    Moreover, for some constants $c_2, C_4,$ assume that $e^{-c_2 d} \le \gamma \le c_2 \sqrt{d}$, and $d \ge C_4$.
    Then, under the prior $p_0 \sim W_d^{-1}(\frac{I}{2d}, 2d)$, there exists an algorithm $\overline{M}$ that is $(\eps, \delta)$-DP, uses $O(n \log(d/\gamma))$ samples, and satisfies \eqref{eq:assumption-stronger}.
\end{lemma}

\begin{proof}
For some parameter $L = O(\log d/\gamma)$, we can obtain $n \cdot L$ samples, which are split into $L$ batches $X^{(1)}, \dots, X^{(L)}$ of $n$ points each. We can compute $M(X^{(i)})$ for each $i$, and if $\|\Sigma\|_{op} \le 10$, from these points we can algorithmically find a value $\tilde{\Sigma}$ which is within distance $2 \gamma$ in Frobenius distance from $\Sigma$ with probability at least $1-e^{-\Omega(L)}$. 

To explain how to find this $\tilde{\Sigma}$, note that by a Chernoff bound, with probability at least $1-e^{-\Omega(L)}$, at least $\frac{2}{3} \cdot L$ values $M(X^{(i)})$ are within $\gamma$ in Frobenius distance from $\Sigma$. In this case, taking the coordinate-wise median of the data points $\{M(X^{(i)})\}_{i=1}^L$ is known to satisfy the guarantees~\cite{Narayanan18}.

Our final modified algorithm $\overline{M}$ takes in $L \cdot n$ samples, computes $\tilde{\Sigma}$ as the coordinate-wise median of $M(X^{(1)}), \dots, M(X^{(L)})$, and equals $\tilde{\Sigma}$ if $\|\tilde{\Sigma}\|_{F} \le 20 \sqrt{d}$, and the $0$ matrix otherwise. Since we used fresh samples in each repetition and the final output is a deterministic function of $M(X^{(1)}), \dots, M(X^{(L)}),$ the algorithm maintains the same privacy guarantees.

Finally, we consider the guarantees we have on $\overline{M}(X)$. If $\|\Sigma\|_{op} \le 10$, then with probability at least $1-e^{-\Omega(L)}$, $\|\tilde{\Sigma}-\Sigma\|_F \le 2 \gamma$. In this case, $\|\tilde{\Sigma}\|_{F} \le \sqrt{d} \cdot \|\Sigma\|_{op}+2\gamma \le 20 \sqrt{d}$, assuming $\gamma \le \sqrt{d}$. This means that $\overline{M}(X) = \tilde{\Sigma}$, so $\|\overline{M}(X)-\Sigma\|_F \le 2 \gamma$. Moreover, regardless of $\Sigma$ and the randomness of $X$ and $M$, we always have $\|\overline{M}(X)\|_{F} \le 20 \sqrt{d}$, so $\|\overline{M}(X)-\Sigma\|_{F} \le 20 \sqrt{d} + \|\Sigma\|_F \le \sqrt{d} \cdot (20+\|\Sigma\|_{op})$.
Therefore, if $\|\Sigma\|_{op} \le 10,$ we can bound 
\begin{align*}
    \BE_{X, \overline{M}}[\|\overline{M}(X)-\Sigma\|_F^4|\Sigma] 
    &\le (2 \gamma)^4 + e^{-\Omega(L)} \cdot (\sqrt{d}(\|\Sigma\|_{op}+20))^4 \\
    &\le O(\gamma^4 + e^{-\Omega(L)} d^2).
\end{align*}
Since we set $L$ to be a sufficiently large multiple of $\log (\frac{d}{\gamma})$ (note that even if $\gamma > 1$, $\frac{d}{\gamma} \ge \sqrt{d}$ still holds), we have that $e^{-\Omega(L)} \le \frac{\gamma^4}{d^2}$, so $\BE_{X, \overline{M}}[\|\overline{M}(X)-\Sigma\|_F^4|\Sigma] \le O(\gamma^4)$ if $\|\Sigma\|_{op} \le 10$.

Alternatively, if $\|\Sigma\|_{op} \ge 10,$ we still have that $\|\overline{M}(X)-\Sigma\|_{F} \le \sqrt{d}(\|\Sigma\|_{op} + 20) \le 3 \sqrt{d} \cdot \|\Sigma\|_{op}$. Thus, $\|\overline{M}(X)-\Sigma\|_{F}^4 \le O(d^2 \cdot \|\Sigma\|_{op}^4)$.

Hence, $\BE_{X, \overline{M}}(\|\overline{M}(X)-\Sigma\|_F^4) \le O(\gamma^4) + O(d^2) \cdot \BE[\|\Sigma\|_{op}^4 \cdot \BI[\|\Sigma\|_{op} \ge 10]]$. By Cauchy-Schwarz, we can bound 
\[\BE[\|\Sigma\|_{op}^4 \cdot \BI[\|\Sigma\|_{op} \ge 10]] \le \sqrt{\BE[\|\Sigma\|_{op}^8] \cdot \BP(\|\Sigma\|_{op} \ge 10)}.\]
By our assumption on the prior $p_0 \sim W_d^{-1}(\frac{I}{m}, m)$, Lemma~\ref{lem:wishart-operator-norm-bound} implies that $\BE[\|\Sigma\|_{op}^8] \le O(1)$ and $\BP(\|\Sigma\|_{op} \ge 10) \le (e^2/10)^{d/2} \le e^{-\Omega(d)}$. Hence, we can bound $\BE_{X, \overline{M}}(\|\overline{M}(X)-\Sigma\|^4) \le O(\gamma^4 + e^{-\Omega(d)}) \le O(\gamma^4)$, assuming that $\gamma \ge e^{-c_2 d}$ for some small constant $c_2$.
\end{proof}

\end{document}